\documentclass[11pt]{article}
\usepackage{amsmath,amsthm,amsfonts,amssymb}
\usepackage{mathrsfs}
\usepackage[colorlinks,linkcolor=blue,anchorcolor=blue,citecolor=blue]{hyperref}
\usepackage[numbers,sort&compress]{natbib}
\usepackage[margin=1.15in]{geometry}
\numberwithin{equation}{section}
\DeclareMathOperator{\tr}{tr}
\DeclareMathOperator{\ind}{ind}
\DeclareMathOperator{\nul}{nul}
\DeclareMathOperator{\divg}{div}
\DeclareMathOperator{\image}{im}
\DeclareMathOperator{\mult}{mult}
\DeclareMathOperator{\proj}{proj}
\DeclareMathOperator{\spec}{spec}
\DeclareMathOperator{\rank}{rank}
\DeclareMathOperator{\diag}{diag}

\theoremstyle{plain}
\newtheorem{theorem}{Theorem}[section]
\newtheorem{lemma}[theorem]{Lemma}
\newtheorem{proposition}[theorem]{Proposition}
\newtheorem{corollary}[theorem]{Corollary}
\theoremstyle{definition}
\newtheorem{definition}[theorem]{Definition}
\theoremstyle{remark}
\newtheorem{remark}[theorem]{Remark}
\newtheorem{example}[theorem]{Example}

\allowdisplaybreaks

\title{\bf Instability and Morse Index of\\ Exponentially Subelliptic Harmonic Maps}
\author{Xin Huang\thanks{The author was partially supported by the Basic Research Program of Jiangsu (Grant No. BK20250730).}}
\date{}

\begin{document}
\maketitle

\begin{abstract}
Let $M$ be a closed manifold with a bracket-generating distribution $H$. We study the Morse index of exponentially subelliptic harmonic maps. After establishing finiteness, we estimate the index through two test spaces. The extrinsic one averages the index form over projected ambient parallel frames, reducing the estimate to a trace identity involving an extrinsic scalar with a quartic exponential term. The intrinsic one uses parallel sections of the pullback bundle, on which the index form descends to a coupled quadratic form that is exact when a global parallel frame exists. The bounds apply to round spheres, products of spheres, and convex hypersurfaces. On a closed Heisenberg nilmanifold, the intrinsic reduction yields a spectral formula for the index.

\smallskip
\noindent\textbf{Keywords:} Exponentially subelliptic harmonic map, Morse index, sub-Riemannian geometry, stability, Heisenberg nilmanifold

\smallskip
\noindent\textbf{MSC:} 58E20, 53C17, 35H20
\end{abstract}

\section{Introduction}\label{sec:intro}
Exponentially harmonic maps were introduced by Eells and Lemaire \cite{EL92} as critical points of the exponential energy, and their variational properties have been investigated in a number of settings \cite{Ch16,LiuJC08}. For maps from a compact Riemannian manifold into the unit sphere $S^n$, Chiang \cite{Ch20} obtained the lower bound $\ind(f)\ge n+1$ under a positivity assumption on the exponential stress-energy tensor, using conformal vector fields of the sphere as variations. In the subelliptic setting, Chiang, Dragomir and Esposito established a first variation formula and a partial regularity theorem for maps from the Heisenberg group into spheres \cite{CDE19}. They subsequently derived the second variation formula in the pseudohermitian setting and studied the associated Jacobi operator \cite{CDE20}. On bounded domains satisfying suitable Poincar\'e and compactness assumptions, they obtained a discrete Dirichlet spectrum and a spectral characterization of stability. Earlier stability results for pseudoharmonic maps are due to Barletta and Dragomir \cite{BD04}. For the horizontal Dirichlet energy on general bracket-generating manifolds, Dong \cite{Don21} developed a theory of subelliptic harmonic maps, while Chong, Dong and Yang \cite{CDY24} studied stability in the presence of a potential.

The question addressed in this paper is the quantitative estimate of the Morse index of an exponentially subelliptic harmonic map $f$. Since the index form is degenerate elliptic in the sub-Riemannian framework, we first prove in Proposition \ref{prop:finiteness} that both the index and the nullity are finite. The proof relies on the compactness of the horizontal Sobolev embedding and is carried out at the level of quadratic forms. The index is then estimated by the elementary principle that $\ind(f)\ge\ind(I_H|_V)$ for every space $V$ of test fields, with equality when $V$ exhausts all sections (Lemma \ref{lem:test-space}). This principle is only an organizing device. The content of the paper lies in the two choices of $V$ studied below, which are of complementary character.

The first test space is extrinsic. If $N$ is isometrically immersed in $\mathbb R^{n+p}$, the tangential components $E_a^\top$ of an ambient parallel orthonormal frame span a test space of dimension at most $n+p$. This is the parallel-frame averaging construction of Howard and Wei \cite{How86}. Related spherical test fields were used by Leung \cite{Leu82} for harmonic maps and by Chiang \cite{Ch20} for exponentially harmonic maps; Liu \cite{LiuJC08} applied the extrinsic method to exponentially harmonic maps into convex hypersurfaces. The space exists for every map. We parametrize it by the coefficient space $\mathbb R^{n+p}$ and use the trace of the resulting coefficient matrix rather than diagonalizing it, which can be computed in closed form. 

\begin{theorem}\label{thm:A}
Let $\iota:N^n\to\mathbb R^{n+p}$ be an isometric immersion and let $f:(M,H,g_H)\to N$ be an exponentially subelliptic harmonic map from a closed manifold $M$. Then
\[
\sum_{a=1}^{n+p}I_H\bigl(E_a^\top,E_a^\top\bigr)=\int_Me^{c}\,\widetilde{\mathscr Q}_N(f)\,dv_g=:\widetilde{\mathcal T}(f),
\]
where $\widetilde{\mathscr Q}_N$ is the extrinsic scalar of Definition \ref{def:howard}. In particular $f$ is unstable if $\widetilde{\mathcal T}(f)<0$.
\end{theorem}

Because the exponential energy is not quadratic in $df$, its second variation contains the term $\mathcal F(W)^2$ with $\mathcal F(W)=\sum_i\langle\tilde\nabla_{e_i}W,df(e_i)\rangle$. Averaging this term produces in $\widetilde{\mathscr Q}_N$ the quartic contribution $\bigl|\sum_iB(df(e_i),df(e_i))\bigr|^2$, which is absent for the Dirichlet energy. This term is responsible for the smallness hypotheses of Section \ref{sec:targets}. Moreover, $\widetilde{\mathscr Q}_N$ has to be evaluated on the whole horizontal differential and cannot be controlled one direction at a time (Example \ref{ex:counterexample}).

The coefficient matrix $\mathbb I=\bigl(I_H(E_a^\top,E_b^\top)\bigr)$ has quadratic form $\lambda\mapsto I_H(P\lambda,P\lambda)$, where $P\lambda$ is the tangential projection along $f$ of a constant ambient vector. Hence its eigenvalues are bounded below by a curvature bound for $N$. Combined with Theorem \ref{thm:A}, this gives the following.

\begin{theorem}\label{thm:B}
In the situation of Theorem \ref{thm:A}, assume moreover that $M$ is connected, that $H$ is bracket-generating, and that $K^N\le\kappa$ along $f(M)$ for some $\kappa>0$. If $\widetilde{\mathcal T}(f)<0$, then $C_0:=\kappa\int_Me^{c}|df_H|^2\,dv_g>0$ and
\[
\ind(f)\ \ge\ \Bigl\lceil\frac{|\widetilde{\mathcal T}(f)|}{C_0}\Bigr\rceil .
\]
\end{theorem}

For round spheres the quotient in Theorem \ref{thm:B} is evaluated exactly and the radius cancels: $\widetilde{\mathcal T}(f)/C_0=\bar s_w-(n-2)$, where $\bar s_w$ is the mean of $|df_H|^2$ with respect to the measure $e^{|df_H|^2/2}|df_H|^2dv_g$. Hence every nonconstant exponentially subelliptic harmonic map $f:M\to S^n(r)$ with $n\ge3$ satisfies $\ind(f)\ge\max\{0,\lceil(n-2)-\bar s_w\rceil\}$ (Corollary \ref{cor:index-sphere}). Analogous statements, together with rigidity theorems, hold for products of spheres and for convex hypersurfaces.

The second test space is intrinsic. Let $P(f)$ be the space of parallel sections of $f^{-1}TN$. When it is nonzero, choose a pointwise orthonormal basis $W_1,\dots,W_k$ and consider the sections $\sum_\alpha v_\alpha W_\alpha$ with arbitrary $v_\alpha\in C^\infty(M)$. On these sections the index form reduces to a coupled quadratic form on $\mathbb R^k$-valued functions.

\begin{theorem}\label{thm:C}
Let $f$ be an exponentially subelliptic harmonic map from a closed manifold. Put $\xi_\alpha:=\sum_i\langle W_\alpha,df(e_i)\rangle e_i\in\Gamma(H)$ and $\rho_{\alpha\beta}:=\sum_i\tilde R_N(df(e_i),W_\alpha,W_\beta,df(e_i))$. For $v=(v_1,\dots,v_k)\in C^\infty(M;\mathbb R^k)$,
\[
I_H\Bigl(\sum_\alpha v_\alpha W_\alpha,\sum_\beta v_\beta W_\beta\Bigr)
=\int_Me^{c}\Bigl[\Bigl(\sum_\alpha\langle\nabla^Hv_\alpha,\xi_\alpha\rangle\Bigr)^2+\sum_\alpha|\nabla^Hv_\alpha|^2-\sum_{\alpha,\beta}\rho_{\alpha\beta}v_\alpha v_\beta\Bigr]dv_g=:\mathcal Q_f(v).
\]
Consequently $\ind(f)\ge\ind(\mathcal Q_f)$. If $\dim P(f)=n$, that is, if $f^{-1}TN$ admits a global parallel orthonormal frame, then $\ind(f)=\ind(\mathcal Q_f)$ and $\nul(f)=\nul(\mathcal Q_f)$.
\end{theorem}

The two test spaces trade generality against precision. The first always exists but is used only through the trace of its coefficient matrix. The second may be trivial, but when the pullback bundle admits a global parallel orthonormal frame it captures the whole index. This case is realized in Section \ref{sec:heis}. On a closed Heisenberg nilmanifold $M=\Gamma\backslash\mathbb H^3$ we consider nonconstant exponentially subelliptic harmonic maps into $S^n(r)$ with constant horizontal energy density whose images are great circles. Such maps go back to a construction of Duan reported in \cite{CDE20}. There $\mathcal Q_f$ splits into a nonnegative tangential block and $n-1$ copies of the scalar form of $\mathcal L-\lambda$, where $\mathcal L$ is the sub-Laplacian, and the index is computed from the spectrum of $\mathcal L$.

\begin{theorem}\label{thm:D}
Let $n\ge2$, let $a,b\in\mathbb Z$ with $(a,b)\ne(0,0)$, and let $f:\Gamma\backslash\mathbb H^3\to S^n(r)$ be the map of Theorem \ref{thm:heis-family}. Put $\lambda:=4\pi^2(a^2+b^2)$ and let $N(\mu)$ be the number of eigenvalues of $\mathcal L$ smaller than $\mu$, counted with multiplicity. Then
\[
\ind(f)=(n-1)\,N(\lambda),
\qquad
\nul(f)=1+(n-1)\,\mult_{\mathcal L}(\lambda),
\]
and $\mult_{\mathcal L}(\lambda)\ge4$.
\end{theorem}

In particular $\ind(f)\ge n-1$ for every radius, and under the parallel trivialization the negative spectral subspace does not depend on $r$. By contrast, the bound of Corollary \ref{cor:index-sphere} decreases in $r$ and is vacuous once $4\pi^2r^2(a^2+b^2)\ge n-2$. The coefficient matrix of the extrinsic test space can be computed explicitly for this family (Remark \ref{rem:ext-matrix}). It has exactly $n-1$ negative eigenvalues for every $r$, which separates the loss due to using only the trace from the loss due to using a finite-dimensional test space. The isometries of the target generate a $(2n-1)$-dimensional subspace of the null space. The full nullity is larger, at least $4n-3$, owing to additional eigenfunctions of $\mathcal L$ with eigenvalue $\lambda$ (Remark \ref{rem:killing-check}). For fixed $n$ and $r$, the Weyl law for sub-Laplacians shows that the index grows like $\lambda^{Q/2}=\lambda^2$, where $Q=4$ is the homogeneous dimension (Remark \ref{rem:weyl}).

The present results extend the setting of \cite{Ch20,CDE20} to closed manifolds with bracket-generating distributions of arbitrary step and a fixed Riemannian extension. The finiteness argument uses closed quadratic forms and requires no boundary conditions. The extrinsic bound of Theorem \ref{thm:B} applies to every isometrically immersed target, and uses the negativity of an explicitly computed averaged trace rather than a pointwise tensor-positivity hypothesis. For spheres, the exponential stress-energy tensor used in \cite[Theorem 2.4]{Ch20} is
\[
S_e(f)=e^{s/2}\bigl(s\,g-(2+s)\,f^*h\bigr),\qquad s=|df|^2 .
\]
Its positive definiteness forces $s>0$ everywhere, since $S_e=0$ where $df=0$. Taking the trace over $(\ker df_x)^\perp$, on which the trace of $f^*h$ equals $s$, gives $\rank(df_x)\,s-(2+s)s>0$, that is, $s<\rank(df_x)-2\le n-2$. Hence, when $H=TM$, the hypothesis of \cite{Ch20} implies $\sup_M|df|^2<n-2$ and thus our condition $\bar s_w<n-2$. Under this hypothesis, the bound $n+1$ of \cite{Ch20} is stronger than ours, but our bound requires only the weaker integral condition. The analogous horizontal condition, positivity of $e^{c}\bigl(|df_H|^2g_H-(2+|df_H|^2)f^*h|_H\bigr)$ on $H$, would force in the same way $0<|df_H|^2<\rank(df_H)-2\le\rank H-2$. It is therefore never satisfied when $\rank H=2$, as for the Heisenberg group and for every three-dimensional contact manifold, whereas the integral condition of Theorem \ref{thm:B} is not subject to this restriction. Finally, for an explicit family on a Heisenberg nilmanifold we compute the index and the nullity exactly. The variation formulas of Section \ref{sec:variation} are recorded in the present framework for completeness.

The paper is organized as follows. Section \ref{sec:variation} develops the variational framework, proves finiteness of the index and nullity, and records the test-space principle and the null directions induced by Killing fields. Sections \ref{sec:ext} and \ref{sec:par} develop the extrinsic and parallel test spaces, respectively, and prove Theorems \ref{thm:A}, \ref{thm:B} and \ref{thm:C}; these sections depend only on Section \ref{sec:variation} and are independent of each other. Section \ref{sec:targets} gives applications to round spheres, products of spheres and convex hypersurfaces. Section \ref{sec:heis} proves Theorem \ref{thm:D} and compares the exact index with the extrinsic estimates. The Appendix identifies the Jacobi differential expression associated with the quadratic form in Proposition \ref{prop:finiteness} and computes its principal symbol.

\section{The Index Form of the Exponential Energy}\label{sec:variation}

\subsection{The Setting}

Throughout, $M$ is a closed connected manifold endowed with a subbundle $H\subset TM$ of rank $m$ and a Riemannian metric $g$. We write $g_H:=g|_H$ and call $g$ a Riemannian extension of $g_H$. The distribution $H$ is bracket-generating if, at every point $x$, iterated brackets of sections of $H$ span $T_xM$. For $x\in M$ let $\mathfrak s(x)$ be the least $r$ such that iterated brackets of length at most $r$ of sections of $H$ span $T_xM$, brackets of length $1$ being the sections of $H$ themselves. Compactness of $M$ implies that $\mathfrak s:=\max_{x\in M}\mathfrak s(x)$ is finite. We call $\mathfrak s$ the step. No equiregularity is assumed. Index ranges are indicated in every summation. The index \(i\) runs over the horizontal directions \(1,\dots,m\), and \(\nu\) over the vertical directions on \(M\). In Section \ref{sec:par}, \(\alpha,\beta\) run over the elements of a basis of \(P(f)\).

Let $\{e_i\}_{i=1}^m$ be a local orthonormal frame of $H$ and $\{e_\nu\}$ one of $\mathcal V:=H^\perp$, and put
\begin{equation}\label{eq:zeta}
\zeta:=\pi_H\Bigl(\sum_\nu\nabla_{e_\nu}e_\nu\Bigr),
\qquad
\nabla^Hu:=\pi_H(\nabla u)=\sum_{i=1}^m(e_iu)\,e_i\quad(u\in C^\infty(M)).
\end{equation}

\begin{remark}\label{rem:extension}
The orthogonal complement $\mathcal V$, the volume form $dv_g$, the Levi-Civita connection $\nabla$, the horizontal projection $\pi_H$ and the field $\zeta$ of \eqref{eq:zeta} depend on $g$ and not on the pair $(H,g_H)$ alone. Since $dv_g$ enters the definition of $E_H$, so do the functional itself, its critical points and their indices. All results below are therefore statements about the triple $(M,H,g)$. When $M$ carries a distinguished extension, as for the left-invariant metric on the Heisenberg nilmanifold of Section \ref{sec:heis}, that extension is understood.
\end{remark}

Let $(N^n,h)$ be a Riemannian manifold with Levi-Civita connection $\tilde\nabla$, and let $f:M\to N$ be smooth.  The indices \(A,B\) run over \(1,\dots,n\) for the tangent directions of \(N\), \(\gamma,\delta\) over \(1,\dots,p\) for the normal directions of an immersion of \(N\), and \(a,b\) over \(1,\dots,n+p\) for the ambient directions. We use the conventions
\[
\tilde R^N(X,Y)Z=\tilde\nabla_X\tilde\nabla_YZ-\tilde\nabla_Y\tilde\nabla_XZ-\tilde\nabla_{[X,Y]}Z,
\qquad
\tilde R_N(X,Y,Z,T):=\bigl\langle\tilde R^N(X,Y)Z,T\bigr\rangle,
\]
and write $\tilde\nabla$ also for the induced connection on $f^{-1}TN$. Put $df_H:=df\circ\pi_H$ and
\begin{equation}\label{eq:cdef}
c:=\tfrac12|df_H|^2,
\qquad
|\nabla_HW|^2:=\sum_{i=1}^m\bigl|\tilde\nabla_{e_i}W\bigr|^2,
\qquad
\mathcal R_T(W,W'):=\sum_{i=1}^m\tilde R_N\bigl(df(e_i),W,W',df(e_i)\bigr).
\end{equation}
The bilinear form $\mathcal R_T$ is symmetric. Indeed, the pair symmetry $\tilde R_N(X,Y,Z,T)=\tilde R_N(Z,T,X,Y)$ followed by the two skew-symmetries gives
\[
\tilde R_N(X,W,W',X)=\tilde R_N(W',X,X,W)=\tilde R_N(X,W',W,X).
\]
Moreover $\mathcal R_T$ depends on $df$ only through the symmetric tensor
\begin{equation}\label{eq:T-tensor}
T_x:=\sum_{i=1}^mdf(e_i)\otimes df(e_i),\qquad T_x(Z)=\sum_{i=1}^m\bigl\langle df(e_i),Z\bigr\rangle df(e_i),
\end{equation}
a self-adjoint positive semidefinite endomorphism of $T_{f(x)}N$ with $\tr T_x=|df_H|^2(x)$. The tensor $T$ does not depend on the horizontal frame. If $e_i'=\sum_jO_{ij}e_j$ with $O\in O(m)$, then $\sum_iO_{ij}O_{ik}=\delta_{jk}$ leaves $T$ unchanged.

The horizontal second fundamental form of $f$ is
\begin{equation}\label{eq:betaH}
\beta_H(f)(X,Y):=\tilde\nabla_Y\bigl(df_H(X)\bigr)-df_H(\nabla_YX),\qquad X,Y\in\Gamma(TM),
\end{equation}
so that
\begin{equation}\label{eq:betaH-diagonal}
\tilde\nabla_{e_i}df(e_i)=\beta_H(f)(e_i,e_i)+df\bigl(\pi_H(\nabla_{e_i}e_i)\bigr),\qquad i=1,\dots,m .
\end{equation}
Although $\beta_H(f)$ is defined on all of $TM\times TM$ and is not symmetric, only its horizontal diagonal values occur below, through \eqref{eq:betaH-diagonal}.

\begin{lemma}[Horizontal constancy]\label{lem:horiz-const}
Let $H$ be bracket-generating and let $f:M\to N$ be smooth with $df_H\equiv0$. Then $f$ is constant. Equivalently, if $f$ is nonconstant then $\{df_H\ne0\}$ is a nonempty open subset of $M$.
\end{lemma}

\begin{proof}
Let $\gamma:[0,1]\to M$ be piecewise $C^1$ with $\dot\gamma(t)\in H_{\gamma(t)}$ for almost every $t$. Then
\[
\frac{d}{dt}\bigl(f\circ\gamma\bigr)(t)=df_{\gamma(t)}\bigl(\dot\gamma(t)\bigr)=df_H\bigl(\dot\gamma(t)\bigr)=0
\]
for almost every $t$, so $f\circ\gamma$ is constant. By the Chow--Rashevskii theorem \cite{Cho39,Mon02}, any two points of the connected manifold $M$ are joined by such a curve, whence $f$ is constant. The last assertion follows from the continuity of $|df_H|$.
\end{proof}

\subsection{Weighted Integration by Parts}

\begin{lemma}\label{lem:2.1}
Let $\psi\in C^\infty(M)$ and let $\alpha$ be a horizontal one-form, that is, $\alpha(X)=\alpha(\pi_HX)$ for all $X$. Then
\begin{equation}\label{eq:ibp}
\int_M\psi\Bigl[\sum_{i=1}^m e_i\alpha(e_i)-\sum_{i=1}^m\alpha\bigl(\pi_H(\nabla_{e_i}e_i)\bigr)\Bigr]dv_g
=\int_M\alpha\bigl(\psi\zeta-\nabla^H\psi\bigr)\,dv_g .
\end{equation}
\end{lemma}

\begin{proof}
Let $\alpha^\sharp$ be the metric dual of $\alpha$. Computing the divergence in the orthonormal frame $\{e_i\}\cup\{e_\nu\}$,
\begin{align}
\divg(\alpha^\sharp)
&=\sum_{i=1}^m\bigl\langle\nabla_{e_i}\alpha^\sharp,e_i\bigr\rangle+\sum_\nu\bigl\langle\nabla_{e_\nu}\alpha^\sharp,e_\nu\bigr\rangle\notag\\
&=\sum_{i=1}^m\bigl[e_i\alpha(e_i)-\alpha(\nabla_{e_i}e_i)\bigr]+\sum_\nu\bigl[e_\nu\alpha(e_\nu)-\alpha(\nabla_{e_\nu}e_\nu)\bigr]\notag\\
&=\sum_{i=1}^m\bigl[e_i\alpha(e_i)-\alpha\bigl(\pi_H(\nabla_{e_i}e_i)\bigr)\bigr]-\alpha(\zeta).\label{eq:div-alpha}
\end{align}
In the last step we used $\alpha(e_\nu)=0$, $\alpha(Y)=\alpha(\pi_HY)$, and $\sum_\nu\alpha(\nabla_{e_\nu}e_\nu)=\alpha(\zeta)$. Furthermore $\alpha(\nabla\psi)=\alpha(\nabla^H\psi)$, so that
\[
\divg\bigl(\psi\alpha^\sharp\bigr)=\psi\,\divg(\alpha^\sharp)+\alpha\bigl(\nabla^H\psi\bigr).
\]
Integrating over the closed manifold $M$ gives $\int_M\psi\,\divg(\alpha^\sharp)\,dv_g=-\int_M\alpha(\nabla^H\psi)\,dv_g$, and substituting \eqref{eq:div-alpha} yields \eqref{eq:ibp}.
\end{proof}

The codifferential formula behind \eqref{eq:div-alpha} was previously obtained in \cite[Lemma 2.2]{CDY24} in an equivalent form. The computation is reproduced below for completeness.

\begin{corollary}\label{cor:2.2}
For every $U\in\Gamma(f^{-1}TN)$,
\begin{multline}\label{eq:ibp-U}
\int_M e^{c}\Bigl[\sum_{i=1}^m e_i\bigl\langle U,df_H(e_i)\bigr\rangle-\sum_{i=1}^m\bigl\langle U,df_H\bigl(\pi_H(\nabla_{e_i}e_i)\bigr)\bigr\rangle\Bigr]dv_g\\
=\int_M e^{c}\bigl\langle U,\,df(\zeta)-df_H(\nabla^Hc)\bigr\rangle\,dv_g .
\end{multline}
\end{corollary}

\begin{proof}
Apply Lemma \ref{lem:2.1} with $\psi=e^{c}$ and the horizontal one-form $\alpha(X):=\langle U,df_H(X)\rangle$, noting that $\nabla^He^{c}=e^{c}\nabla^Hc$ and $df_H(\zeta)=df(\zeta)$.
\end{proof}

\subsection{First and Second Variation}

For $f:M\to N$ put
\[
E_H(f):=\int_Me^{c}\,dv_g=\int_Me^{|df_H|^2/2}\,dv_g .
\]

\begin{definition}
A critical point of $E_H$ is called an \emph{exponentially subelliptic harmonic map}.
\end{definition}

Every $W\in\Gamma(f^{-1}TN)$ is the variation field of a smooth variation of $f$. Indeed, putting $f_t(x):=\exp_{f(x)}(tW(x))$, the compactness of $M$ yields $\epsilon>0$ such that $(x,t)\mapsto f_t(x)$ is smooth on $M\times(-\epsilon,\epsilon)$, with $f_0=f$ and $\partial_tf_t|_{t=0}=W$. The variational conditions below may therefore be tested against arbitrary smooth sections.

\begin{theorem}\label{th:2.3}
Let $f_t$ be a smooth variation of $f$ with $W:=\partial_tf_t|_{t=0}$. Then
\[
\frac{d}{dt}E_H(f_t)\Bigr|_{t=0}=-\int_M e^{c}\bigl\langle W,\tau_H(f)\bigr\rangle\,dv_g,
\]
where
\begin{equation}\label{eq:tension}
\tau_H(f):=\sum_{i=1}^m\beta_H(f)(e_i,e_i)-df(\zeta)+df_H(\nabla^Hc)
\end{equation}
is the exponential tension field. Hence $f$ is exponentially subelliptic harmonic if and only if $\tau_H(f)=0$.
\end{theorem}

\begin{proof}
Let $F(x,t):=f_t(x)$ and $W_t:=dF(\partial_t)$, so that $W_0=W$. Extend $\{e_i\}$ to $M\times(-\epsilon,\epsilon)$ independently of $t$, so that $[e_i,\partial_t]=0$. Since $\tilde\nabla$ is torsion free, $\tilde\nabla_{\partial_t}dF(e_i)=\tilde\nabla_{e_i}W_t$, and differentiating under the integral sign gives
\begin{equation}\label{eq:2.10}
\frac{d}{dt}E_H(f_t)=\int_M e^{\frac12|d(f_t)_H|^2}\sum_{i=1}^m\bigl\langle\tilde\nabla_{e_i}W_t,dF(e_i)\bigr\rangle\,dv_g .
\end{equation}
At $t=0$, metric compatibility of $\tilde\nabla$ and \eqref{eq:betaH-diagonal} give
\begin{align}
\sum_{i=1}^m\bigl\langle\tilde\nabla_{e_i}W,df(e_i)\bigr\rangle
={}&\sum_{i=1}^m e_i\bigl\langle W,df_H(e_i)\bigr\rangle-\Bigl\langle W,\sum_{i=1}^m\tilde\nabla_{e_i}df(e_i)\Bigr\rangle\notag\\
={}&\sum_{i=1}^m e_i\bigl\langle W,df_H(e_i)\bigr\rangle-\sum_{i=1}^m\bigl\langle W,df_H\bigl(\pi_H(\nabla_{e_i}e_i)\bigr)\bigr\rangle\notag\\
&-\Bigl\langle W,\sum_{i=1}^m\beta_H(f)(e_i,e_i)\Bigr\rangle .\label{eq:F-expand}
\end{align}
Multiply by $e^{c}$ and integrate. Corollary \ref{cor:2.2} with $U=W$ turns the first two terms into $\int_Me^{c}\langle W,df(\zeta)-df_H(\nabla^Hc)\rangle\,dv_g$, which gives the stated formula. The final assertion follows from the arbitrariness of $W$.
\end{proof}

\begin{corollary}\label{cor:F-vanishes}
Let $f$ be exponentially subelliptic harmonic and put
\[
\mathcal F(U):=\sum_{i=1}^m\bigl\langle\tilde\nabla_{e_i}U,df(e_i)\bigr\rangle,\qquad U\in\Gamma(f^{-1}TN).
\]
Then $\int_Me^{c}\,\mathcal F(U)\,dv_g=0$ for every $U\in\Gamma(f^{-1}TN)$.
\end{corollary}

\begin{proof}
Multiply \eqref{eq:F-expand}, with $W$ replaced by $U$, by $e^{c}$ and integrate. By Corollary \ref{cor:2.2} the right-hand side becomes $-\int_Me^{c}\langle U,\tau_H(f)\rangle\,dv_g$, which vanishes since $\tau_H(f)=0$.
\end{proof}

\begin{theorem}\label{th:2.4}
Let $f$ be exponentially subelliptic harmonic, and let $h:M\times(-\epsilon,\epsilon)^2\to N$ be smooth with $h(\cdot,0,0)=f$. Put $W:=\partial_th|_{(0,0)}$ and $W':=\partial_sh|_{(0,0)}$. Then
\begin{equation}\label{eq:2ndvar}
\frac{\partial^2}{\partial t\,\partial s}E_H\bigl(h(\cdot,t,s)\bigr)\Bigr|_{(0,0)}
=\int_Me^{c}\Bigl[\mathcal F(W)\mathcal F(W')+\sum_{i=1}^m\bigl\langle\tilde\nabla_{e_i}W,\tilde\nabla_{e_i}W'\bigr\rangle-\mathcal R_T(W,W')\Bigr]dv_g .
\end{equation}
\end{theorem}

\begin{proof}
Write $c_h:=\frac12|dh_H|^2$ and extend $\{e_i\}$ independently of $(t,s)$. As in \eqref{eq:2.10},
\[
\partial_sE_H(h)=\int_Me^{c_h}\sum_{i=1}^m\bigl\langle\tilde\nabla_{e_i}\partial_sh,dh(e_i)\bigr\rangle\,dv_g .
\]
Now differentiate in $t$. Since $[e_i,\partial_t]=0$ and $\tilde\nabla$ is torsion free, $\tilde\nabla_{\partial_t}dh(e_i)=\tilde\nabla_{e_i}\partial_th$. Consequently
\[
\partial_tc_h=\sum_{i=1}^m\bigl\langle\tilde\nabla_{e_i}\partial_th,dh(e_i)\bigr\rangle,
\]
which at $(0,0)$ equals $\mathcal F(W)$, and
\[
\partial_t\bigl\langle\tilde\nabla_{e_i}\partial_sh,dh(e_i)\bigr\rangle
=\bigl\langle\tilde\nabla_{\partial_t}\tilde\nabla_{e_i}\partial_sh,dh(e_i)\bigr\rangle+\bigl\langle\tilde\nabla_{e_i}\partial_sh,\tilde\nabla_{e_i}\partial_th\bigr\rangle .
\]
The curvature convention applied to the commuting pair $(e_i,\partial_t)$ gives
\[
\tilde\nabla_{\partial_t}\tilde\nabla_{e_i}\partial_sh=\tilde\nabla_{e_i}Z-\tilde R^N\bigl(dh(e_i),\partial_th\bigr)\partial_sh,
\qquad Z:=\tilde\nabla_{\partial_t}\partial_sh .
\]
Evaluating at $(0,0)$ we obtain
\begin{align}
\frac{\partial^2E_H(h)}{\partial t\,\partial s}\Bigr|_{(0,0)}
={}&\int_Me^{c}\Bigl[\mathcal F(W)\mathcal F(W')+\sum_{i=1}^m\bigl\langle\tilde\nabla_{e_i}W',\tilde\nabla_{e_i}W\bigr\rangle-\mathcal R_T(W,W')\Bigr]dv_g\notag\\
&+\int_Me^{c}\,\mathcal F\bigl(Z|_{(0,0)}\bigr)\,dv_g .
\end{align}
The last integral vanishes by Corollary \ref{cor:F-vanishes}, and \eqref{eq:2ndvar} follows.
\end{proof}

\subsection{The Index Form and its Finiteness}

\begin{definition}\label{def:index-form}
For an exponentially subelliptic harmonic map $f$, the \emph{index form} $I_H$ is the symmetric bilinear form on $\Gamma(f^{-1}TN)$ given by the right-hand side of \eqref{eq:2ndvar}:
\begin{equation}\label{eq:index-polarized}
I_H(W,W'):=\int_Me^{c}\Bigl[\mathcal F(W)\mathcal F(W')+\sum_{i=1}^m\bigl\langle\tilde\nabla_{e_i}W,\tilde\nabla_{e_i}W'\bigr\rangle-\mathcal R_T(W,W')\Bigr]dv_g .
\end{equation}
The map $f$ is \emph{stable} if $I_H(W,W)\ge0$ for all $W$. Its \emph{index} $\ind(f)\in\{0,1,2,\dots\}\cup\{+\infty\}$ is the supremum of the dimensions of the linear subspaces of $\Gamma(f^{-1}TN)$ on which $I_H$ is negative definite.
\end{definition}

By Theorem \ref{th:2.4} with $h(x,t,s)=f_{t+s}(x)$, one has $I_H(W,W)=\frac{d^2}{dt^2}E_H(f_t)|_{t=0}$ for any variation with field $W$. If $K^N\le0$ then $\mathcal R_T(W,W)\le0$, while the first two terms of \eqref{eq:index-polarized} with $W'=W$ are nonnegative. Hence every exponentially subelliptic harmonic map into a manifold of nonpositive sectional curvature is stable.

Let $L^2_c:=L^2(f^{-1}TN,e^{c}dv_g)$ with inner product $\langle\!\langle W,W'\rangle\!\rangle:=\int_Me^{c}\langle W,W'\rangle\,dv_g$, and let $\mathcal H^1_H$ be the completion of $\Gamma(f^{-1}TN)$ with respect to
\[
\|W\|_{\mathcal H^1_H}^2:=\int_Me^{c}\bigl(|\nabla_HW|^2+|W|^2\bigr)\,dv_g .
\]
The same notation is used for scalar and vector-valued functions and for sections of other Riemannian vector bundles with metric connection.

The completion embeds naturally in $L^2$. Indeed, let $W_j$ be smooth with $W_j\to0$ in $L^2$ and $\tilde\nabla_{e_i}W_j\to G_i$ in $L^2$ over the domain of a local frame, for each $i$. For every smooth section $\Phi$ with compact support in that domain, integration by parts gives
\[
\int\bigl\langle\tilde\nabla_{e_i}W_j,\Phi\bigr\rangle\,dv_g=-\int\bigl\langle W_j,\tilde\nabla_{e_i}\Phi+(\divg e_i)\,\Phi\bigr\rangle\,dv_g\longrightarrow0,
\]
so $G_i=0$. Hence the operator $W\mapsto(\tilde\nabla_{e_i}W)_i$ is closable, and $\mathcal H^1_H$ is identified with a subspace of $L^2_c$.

\begin{lemma}\label{lem:compact-embed}
Let $H$ be bracket-generating of step $\mathfrak s$, and let $\mathcal E\to M$ be a Riemannian vector bundle with metric connection $\nabla^{\mathcal E}$. Then the inclusion $\mathcal H^1_H(\mathcal E)\hookrightarrow L^2(\mathcal E)$ is compact.
\end{lemma}

\begin{proof}
Since $M$ is closed, the weight $e^{c}$ is bounded between two positive constants and may be ignored. Here $|\nabla_HW|^2=\sum_i|\nabla^{\mathcal E}_{e_i}W|^2$. Choose finitely many open sets $U_\rho$ covering $M$, each with closure contained in an open set $U_\rho'$ over which $\mathcal E$ has an orthonormal frame $\{\sigma^\rho_l\}_{l=1}^q$, and functions $\chi_\rho\in C^\infty_c(U_\rho)$ with $\sum_\rho\chi_\rho^2=1$. On $U_\rho'$ write $W=\sum_lW^{(l)}\sigma^\rho_l$. Then
\[
\nabla^{\mathcal E}_{e_i}W=\sum_{l=1}^q\bigl(e_iW^{(l)}\bigr)\sigma^\rho_l+\sum_{l=1}^qW^{(l)}\,\nabla^{\mathcal E}_{e_i}\sigma^\rho_l ,
\]
and the coefficients $\nabla^{\mathcal E}_{e_i}\sigma^\rho_l$ are bounded on $\overline U_\rho$. The inequality $|u+v|^2\ge\frac12|u|^2-|v|^2$ gives a constant $C_\rho$ with
\[
|\nabla_HW|^2\ \ge\ \tfrac12\sum_{l=1}^q|\nabla^HW^{(l)}|^2-C_\rho|W|^2\qquad\text{on }\overline U_\rho .
\]
Since $\nabla^H(\chi_\rho W^{(l)})=\chi_\rho\nabla^HW^{(l)}+W^{(l)}\nabla^H\chi_\rho$, it follows that
\[
\|\nabla^H(\chi_\rho W^{(l)})\|_{L^2(M)}+\|\chi_\rho W^{(l)}\|_{L^2(M)}\le C\|W\|_{\mathcal H^1_H}
\]
for every pair $(\rho,l)$. Because $H$ is bracket-generating of step $\mathfrak s$, the subelliptic estimate of Rothschild and Stein \cite{RS76} gives, for scalar functions,
\[
\|u\|_{W^{1/\mathfrak s,2}(M)}\le C\bigl(\|\nabla^Hu\|_{L^2(M)}+\|u\|_{L^2(M)}\bigr),\qquad u\in C^\infty(M),
\]
where $W^{1/\mathfrak s,2}(M)$ is the fractional Sobolev space of $(M,g)$. By density, the estimate extends to the completion. Hence a bounded sequence in $\mathcal H^1_H(\mathcal E)$ yields, for each of the finitely many pairs $(\rho,l)$, a bounded sequence in $W^{1/\mathfrak s,2}(M)$. By Rellich's theorem on the closed manifold $M$, after finitely many extractions all the $\chi_\rho W^{(l)}$ converge in $L^2(M)$. Then $\chi_\rho W=\sum_l(\chi_\rho W^{(l)})\sigma^\rho_l$ converges in $L^2(\mathcal E)$ for each $\rho$, and so does $W=\sum_\rho\chi_\rho(\chi_\rho W)$.
\end{proof}

\begin{proposition}\label{prop:finiteness}
Let $f$ be an exponentially subelliptic harmonic map from a closed manifold with $H$ bracket-generating. Put
\[
\kappa_f:=\max\Bigl\{0,\ \sup\bigl\{K^N(\pi):\pi\subset T_qN\ \text{a plane},\ q\in f(M)\bigr\}\Bigr\},
\qquad
C_1:=\kappa_f\sup_M|df_H|^2 ,
\]
with the convention $\kappa_f=0$ if $\dim N=1$. Then:
\begin{enumerate}
\item[\rm(i)] $I_H(W,W)\ge\int_Me^{c}|\nabla_HW|^2dv_g-C_1\|W\|^2_{L^2_c}$ for all $W$, and $I_H$ extends to a bounded symmetric bilinear form on $\mathcal H^1_H$;
\item[\rm(ii)] $a(W,W'):=I_H(W,W')+C_1\langle\!\langle W,W'\rangle\!\rangle$ is a closed nonnegative form on $L^2_c$ with form domain $\mathcal H^1_H$, and the associated nonnegative self-adjoint operator $A$ has compact resolvent;
\item[\rm(iii)] $\ind(f)=\#\{j:\mu_j<C_1\}<\infty$, where $\mu_0\le\mu_1\le\cdots\to\infty$ are the eigenvalues of $A$ repeated according to multiplicity, and the nullity
\[
\nul(f):=\dim\bigl\{W\in\mathcal H^1_H:\ I_H(W,W')=0\ \text{for all }W'\in\mathcal H^1_H\bigr\}=\dim\ker(A-C_1)
\]
is finite.
\end{enumerate}
\end{proposition}

\begin{proof}
(i) Since $M$ is closed, $f(M)$ is compact, so $\kappa_f<\infty$ and $\sup_M|df_H|^2<\infty$. For each $i$,
\[
\tilde R_N\bigl(df(e_i),W,W,df(e_i)\bigr)\le\kappa_f\bigl(|df(e_i)|^2|W|^2-\langle df(e_i),W\rangle^2\bigr)\le\kappa_f|df(e_i)|^2|W|^2 ,
\]
whence $\mathcal R_T(W,W)\le C_1|W|^2$; when $\dim N=1$ the curvature term vanishes. Dropping the nonnegative term $\mathcal F(W)^2$ in \eqref{eq:index-polarized} gives the lower bound. For boundedness, the Cauchy--Schwarz inequality gives $\mathcal F(W)^2\le|df_H|^2|\nabla_HW|^2$, and $|\mathcal R_T(W,W')|\le C|W||W'|$ with $C$ depending on $\sup_{f(M)}|\tilde R^N|$ and $\sup_M|df_H|^2$. Hence $|I_H(W,W')|\le C\|W\|_{\mathcal H^1_H}\|W'\|_{\mathcal H^1_H}$.

(ii) By (i), $a(W,W)\ge\int_Me^{c}|\nabla_HW|^2dv_g\ge0$ and
\[
\|W\|^2_{\mathcal H^1_H}\ \le\ a(W,W)+\|W\|^2_{L^2_c}\ \le\ C\|W\|^2_{\mathcal H^1_H},
\]
so the form norm of $a$ is equivalent to $\|\cdot\|_{\mathcal H^1_H}$. As $\mathcal H^1_H$ is complete and embedded in $L^2_c$, $a$ is a closed nonnegative form. By the representation theorem for closed forms there is a unique nonnegative self-adjoint operator $A$ on $L^2_c$ with form domain $\mathcal H^1_H$ and $a(W,W')=\langle\!\langle AW,W'\rangle\!\rangle$ for $W$ in its domain. Since the embedding $\mathcal H^1_H\hookrightarrow L^2_c$ is compact by Lemma \ref{lem:compact-embed}, $A$ has compact resolvent, hence discrete spectrum with finite multiplicities.

(iii) On $\mathcal H^1_H$ we have $I_H(W,W)=a(W,W)-C_1\|W\|^2_{L^2_c}$. By the min--max principle, the maximal dimension $k$ of a subspace of $\mathcal H^1_H$ on which $I_H$ is negative definite equals $\#\{j:\mu_j<C_1\}$, a finite number, and the radical of $I_H$ on $\mathcal H^1_H$ is $\ker(A-C_1)$, of finite dimension.

It remains to see that restricting to smooth sections does not change the index. Since $\Gamma(f^{-1}TN)\subset\mathcal H^1_H$, we have $\ind(f)\le k$. If $k=0$ there is nothing more to prove, so assume $k\ge1$. Let $\mathcal W\subset\mathcal H^1_H$ be a $k$-dimensional subspace on which $I_H$ is negative definite, with basis $W_1,\dots,W_k$ normalized in $\mathcal H^1_H$. The function
\[
\lambda\longmapsto I_H\Bigl(\sum_l\lambda_lW_l,\sum_l\lambda_lW_l\Bigr)
\]
is continuous and strictly negative on the compact unit sphere of $\mathbb R^k$, hence bounded above there by some $-\delta<0$. By definition of $\mathcal H^1_H$ there are smooth sections $W_l^\varepsilon$ with $\|W_l^\varepsilon-W_l\|_{\mathcal H^1_H}<\varepsilon$. By the boundedness in (i), the same expression formed with the $W_l^\varepsilon$ is bounded above by $-\delta/2$ on the unit sphere once $\varepsilon$ is small. In particular the $W_l^\varepsilon$ are linearly independent and span a $k$-dimensional subspace of $\Gamma(f^{-1}TN)$ on which $I_H$ is negative definite. Hence $\ind(f)\ge k$.
\end{proof}

Throughout the paper, the nullity is understood in the quadratic-form sense of Proposition \ref{prop:finiteness}(iii). The relation of $A$ with the Jacobi operator is recorded in Remark \ref{rem:jacobi-link}.

\subsection{Test Spaces and Killing Fields}

For a real vector space $V$ and a symmetric bilinear form $b$ on $V$, let $\ind(b)$ denote the supremum of the dimensions of the subspaces of $V$ on which $b$ is negative definite.

\begin{lemma}[Test spaces]\label{lem:test-space}
Let $V$ be a real vector space and $\Phi:V\to\Gamma(f^{-1}TN)$ an injective linear map. Then
\[
\ind(f)\ \ge\ \ind\bigl(I_H\circ(\Phi\times\Phi)\bigr),
\]
with equality if $\Phi$ is onto.
\end{lemma}

\begin{proof}
If $I_H\circ(\Phi\times\Phi)$ is negative definite on a subspace $V_0\subset V$, then $\Phi$ maps $V_0$ isomorphically onto $\Phi(V_0)$, on which $I_H$ is negative definite. If $\Phi$ is bijective, $\Phi^{-1}$ provides the converse.
\end{proof}

\begin{proposition}[Killing null directions]\label{prop:killing}
Let $f$ be exponentially subelliptic harmonic and let $K$ be a Killing field defined on a neighbourhood of $f(M)$. Then $K\circ f$ lies in the radical of $I_H$ on $\mathcal H^1_H$. Consequently $\nul(f)\ge\dim\{K\circ f:K\in\mathfrak k\}$ for every linear space $\mathfrak k$ of such Killing fields.
\end{proposition}

\begin{proof}
Let $\phi_t$ be the local flow of $K$. Since $f(M)$ is compact, there are a neighbourhood $U$ of $f(M)$ and $\epsilon>0$ such that $\phi_t$ is defined on $U$ for $|t|<\epsilon$. Each $\phi_t$ is an isometry onto its image. Hence, for every smooth $g$ with $g(M)\subset U$, one has $|d(\phi_t\circ g)(X)|=|dg(X)|$ for all $X$, so $|d(\phi_t\circ g)_H|=|dg_H|$ and $E_H(\phi_t\circ g)=E_H(g)$.

Given $W'\in\Gamma(f^{-1}TN)$, put $f_s:=\exp_f(sW')$, which takes values in $U$ for small $|s|$, and $h(x,t,s):=\phi_t(f_s(x))$. Then $E_H(h(\cdot,t,s))=E_H(f_s)$ does not depend on $t$, so its mixed derivative vanishes. Since $\partial_th|_{(0,0)}=K\circ f$ and $\partial_sh|_{(0,0)}=W'$, Theorem \ref{th:2.4} gives $I_H(K\circ f,W')=0$ for all smooth $W'$. By density and Proposition \ref{prop:finiteness}(i), the same holds for all $W'\in\mathcal H^1_H$.
\end{proof}

\section{The Extrinsic Test Space}\label{sec:ext}

This section studies the first test space. It exists for every map into an isometrically immersed target and has at most $n+p$ dimensions. We parametrize it by the coefficient space $\mathbb R^{n+p}$, compute the trace of the resulting coefficient matrix, and bound its eigenvalues. The trace is taken on $\mathbb R^{n+p}$. The projected fields need not form an $L^2_c$-orthonormal basis of their span.

\subsection{The Extrinsic Quantity}

Let $\iota:N^n\to\mathbb R^{n+p}$ be an isometric immersion with second fundamental form
\[
B(X,Y)=\bigl({}^{R}\!\tilde\nabla_XY\bigr)^{\perp},\qquad X,Y\in\Gamma(TN),
\]
where ${}^{R}\!\tilde\nabla$ is the flat connection of $\mathbb R^{n+p}$ and $(\,\cdot\,)^\perp$ the projection onto the normal bundle. The shape operator satisfies $\langle\mathcal A^\xi X,Y\rangle=\langle B(X,Y),\xi\rangle$ for normal $\xi$, and the Gauss equation reads
\begin{equation}\label{eq:gauss}
\tilde R_N(X,Y,Y,X)=\bigl\langle B(X,X),B(Y,Y)\bigr\rangle-\bigl|B(X,Y)\bigr|^2,\qquad X,Y\in T_qN .
\end{equation}
Choose orthonormal bases $\{\varepsilon_A\}_{A=1}^n$ of $T_qN$ and $\{\nu_\gamma\}_{\gamma=1}^p$ of the normal space at $q$, and put
\[
\eta:=\sum_{A=1}^nB(\varepsilon_A,\varepsilon_A),
\qquad
\sigma(X):=\sum_{A=1}^n\bigl|B(X,\varepsilon_A)\bigr|^2 .
\]
We identify $T_qN$ with $d\iota_q(T_qN)\subset\mathbb R^{n+p}$ by means of $d\iota$. For a parallel orthonormal frame $\{E_a\}_{a=1}^{n+p}$ of $\mathbb R^{n+p}$, the tangential component is
\begin{equation}\label{eq:Etop}
E_a^\top(q):=\bigl(d\iota_q\bigr)^{-1}\Bigl[\proj_{d\iota_q(T_qN)}E_a\Bigr]\in T_qN ,
\end{equation}
which is well defined since $d\iota_q$ is injective. Along $f$ we write $E_a^\top$ for the section $E_a^\top\circ f\in\Gamma(f^{-1}TN)$.

\begin{definition}\label{def:howard}
For $X_1,\dots,X_m\in T_qN$ put
\begin{equation}\label{eq:Qdef}
\widetilde{\mathscr Q}_N(X_1,\dots,X_m):=\Bigl|\sum_{i=1}^mB(X_i,X_i)\Bigr|^2-\Bigl\langle\sum_{i=1}^mB(X_i,X_i),\eta\Bigr\rangle+2\sum_{i=1}^m\sigma(X_i),
\end{equation}
and for $f:M\to N$ write $\widetilde{\mathscr Q}_N(f):=\widetilde{\mathscr Q}_N(df(e_1),\dots,df(e_m))$. If $\iota$ is minimal, so that $\eta\equiv0$, then $\widetilde{\mathscr Q}_N\ge0$ and the averaged trace criterion gives no instability conclusion.
\end{definition}

For the Dirichlet energy, the averaged quantity is the Howard--Wei expression $2\sum_i\sigma(df(e_i))-\langle S_f,\eta\rangle$ \cite{How86}, with $S_f$ as in Lemma \ref{lem:frame}. The single-direction specialization of \eqref{eq:Qdef},
\[
q(X):=\widetilde{\mathscr Q}_N(X)=\bigl|B(X,X)\bigr|^2-\bigl\langle B(X,X),\eta\bigr\rangle+2\sigma(X),
\]
contains in addition the quartic term $|B(X,X)|^2$ produced by the exponential factor. When $N=S^n(r)$, the field $E_a^\top$ is the gradient of the restriction of the linear function $\langle x,E_a\rangle$, a conformal field of $S^n(r)$. These are the test fields used by Leung \cite{Leu82} and by Chiang \cite{Ch20}. The conditions $q(X_i)<0$ for every $i$ do not imply $\widetilde{\mathscr Q}_N(X_1,\dots,X_m)<0$, so the quantity must be evaluated on the whole horizontal differential.

\begin{example}\label{ex:counterexample}
Let $N=S^3(1)\subset\mathbb R^4$ with inward unit normal $\nu=-x$, so that $B(X,Y)=\langle X,Y\rangle\nu$, $\eta=3\nu$ and $\sigma(X)=|X|^2$. Then, for $X_1,X_2\in T_qS^3(1)$ and $s:=|X_1|^2+|X_2|^2$,
\[
q(X)=|X|^4-3|X|^2+2|X|^2=|X|^2\bigl(|X|^2-1\bigr),
\qquad
\widetilde{\mathscr Q}_N(X_1,X_2)=s^2-3s+2s=s(s-1).
\]
Choosing $|X_1|^2=|X_2|^2=\frac34$ gives
\[
q(X_1)=q(X_2)=\tfrac34\bigl(\tfrac34-1\bigr)=-\tfrac3{16}<0,
\qquad
\widetilde{\mathscr Q}_N(X_1,X_2)=\tfrac32\bigl(\tfrac32-1\bigr)=\tfrac34>0 .
\]
\end{example}

\begin{lemma}\label{lem:frame-indep}
$\widetilde{\mathscr Q}_N(f)$ is independent of the choice of horizontal frame, and hence is a globally defined function on $M$.
\end{lemma}

\begin{proof}
If $e_i'=\sum_{j=1}^mO_{ij}e_j$ with $O\in O(m)$, then $\sum_{i=1}^mO_{ij}O_{ik}=\delta_{jk}$, so
\[
\sum_{i=1}^mB\bigl(df(e_i'),df(e_i')\bigr)
=\sum_{j,k=1}^m\Bigl(\sum_{i=1}^mO_{ij}O_{ik}\Bigr)B\bigl(df(e_j),df(e_k)\bigr)
=\sum_{j=1}^mB\bigl(df(e_j),df(e_j)\bigr).
\]
Moreover $\sum_{i=1}^m\sigma(df(e_i))=\sum_{\gamma=1}^p\tr\bigl((\mathcal A^{\nu_\gamma})^2T\bigr)$ depends only on the frame-independent tensor $T$ of \eqref{eq:T-tensor}.
\end{proof}

\subsection{The Averaging Identity}

\begin{lemma}\label{lem:frame}
Let $\{E_a\}_{a=1}^{n+p}$ be a parallel orthonormal frame of $\mathbb R^{n+p}$ and write $E_a=E_a^\top+E_a^\perp$ along $N$. Then for all $X\in T_qN$
\begin{equation}\label{eq:weingarten}
\tilde\nabla_XE_a^\top=\mathcal A^{E_a^\perp}(X),
\end{equation}
and, with $S_f:=\sum_{i=1}^mB\bigl(df(e_i),df(e_i)\bigr)$,
\begin{align}
\sum_{a=1}^{n+p}\mathcal F\bigl(E_a^\top\bigr)^2&=\bigl|S_f\bigr|^2,\label{eq:l1}\\
\sum_{a=1}^{n+p}\bigl|\nabla_HE_a^\top\bigr|^2&=\sum_{i=1}^m\sigma\bigl(df(e_i)\bigr),\label{eq:l2}\\
\sum_{a=1}^{n+p}\mathcal R_T\bigl(E_a^\top,E_a^\top\bigr)&=\bigl\langle S_f,\eta\bigr\rangle-\sum_{i=1}^m\sigma\bigl(df(e_i)\bigr).\label{eq:l3}
\end{align}
\end{lemma}

\begin{proof}
Since $E_a$ is constant, ${}^{R}\!\tilde\nabla_XE_a=0$. The Gauss--Weingarten formulas
\[
{}^{R}\!\tilde\nabla_XE_a^\top=\tilde\nabla_XE_a^\top+B(X,E_a^\top),
\qquad
{}^{R}\!\tilde\nabla_XE_a^\perp=-\mathcal A^{E_a^\perp}(X)+\nabla^\perp_XE_a^\perp
\]
give \eqref{eq:weingarten} upon taking tangential parts.

For \eqref{eq:l1}, formula \eqref{eq:weingarten} gives $\langle\tilde\nabla_{e_i}E_a^\top,df(e_i)\rangle=\langle B(df(e_i),df(e_i)),E_a\rangle$ for each $i$. Summing over $i$ gives $\mathcal F(E_a^\top)=\langle S_f,E_a\rangle$, and Parseval's identity for the orthonormal basis $\{E_a\}$ of $\mathbb R^{n+p}$ yields $\sum_a\langle S_f,E_a\rangle^2=|S_f|^2$.

For \eqref{eq:l2}, write $E_a^\perp=\sum_{\gamma=1}^p\langle E_a,\nu_\gamma\rangle\nu_\gamma$, so that by \eqref{eq:weingarten}
\[
\tilde\nabla_{e_i}E_a^\top=\sum_{\gamma=1}^p\langle E_a,\nu_\gamma\rangle\,\mathcal A^{\nu_\gamma}\bigl(df(e_i)\bigr).
\]
Using $\sum_{a=1}^{n+p}\langle E_a,\nu_\gamma\rangle\langle E_a,\nu_\delta\rangle=\delta_{\gamma\delta}$ and expanding in the basis $\{\varepsilon_A\}$,
\begin{align}
\sum_{a=1}^{n+p}\bigl|\tilde\nabla_{e_i}E_a^\top\bigr|^2
&=\sum_{\gamma,\delta=1}^p\delta_{\gamma\delta}\bigl\langle\mathcal A^{\nu_\gamma}(df(e_i)),\mathcal A^{\nu_\delta}(df(e_i))\bigr\rangle
=\sum_{\gamma=1}^p\bigl|\mathcal A^{\nu_\gamma}(df(e_i))\bigr|^2\notag\\
&=\sum_{A=1}^n\sum_{\gamma=1}^p\bigl\langle B(df(e_i),\varepsilon_A),\nu_\gamma\bigr\rangle^2
=\sum_{A=1}^n\bigl|B(df(e_i),\varepsilon_A)\bigr|^2=\sigma\bigl(df(e_i)\bigr).
\end{align}
Summation over $i$ gives \eqref{eq:l2}.

For \eqref{eq:l3}, the Gauss equation \eqref{eq:gauss} with $X=df(e_i)$ and $Y=E_a^\top$ gives
\[
\tilde R_N\bigl(df(e_i),E_a^\top,E_a^\top,df(e_i)\bigr)
=\bigl\langle B(df(e_i),df(e_i)),B(E_a^\top,E_a^\top)\bigr\rangle-\bigl|B(df(e_i),E_a^\top)\bigr|^2 .
\]
Write $E_a^\top=\sum_{A=1}^n\langle E_a,\varepsilon_A\rangle\varepsilon_A$ and use $\sum_{a=1}^{n+p}\langle E_a,\varepsilon_A\rangle\langle E_a,\varepsilon_B\rangle=\delta_{AB}$. Then
\begin{align}
\sum_{a=1}^{n+p}B\bigl(E_a^\top,E_a^\top\bigr)&=\sum_{A,B=1}^n\delta_{AB}B(\varepsilon_A,\varepsilon_B)=\eta,\\
\sum_{a=1}^{n+p}\bigl|B(df(e_i),E_a^\top)\bigr|^2&=\sum_{A=1}^n\bigl|B(df(e_i),\varepsilon_A)\bigr|^2=\sigma\bigl(df(e_i)\bigr).
\end{align}
Summation over $a$ and then over $i$ gives \eqref{eq:l3}.
\end{proof}

\begin{theorem}\label{thm:master}
Let $\iota:N^n\to\mathbb R^{n+p}$ be an isometric immersion and let $f:M\to N$ be exponentially subelliptic harmonic, $M$ closed. Then
\begin{equation}\label{eq:master}
\sum_{a=1}^{n+p}I_H\bigl(E_a^\top,E_a^\top\bigr)=\int_Me^{c}\,\widetilde{\mathscr Q}_N(f)\,dv_g=:\widetilde{\mathcal T}(f).
\end{equation}
If $\widetilde{\mathcal T}(f)<0$, then $I_H(E_{a_0}^\top,E_{a_0}^\top)\le\widetilde{\mathcal T}(f)/(n+p)<0$ for some $a_0$, and $f$ is unstable.
\end{theorem}

\begin{proof}
Take $W=W'=E_a^\top$ in \eqref{eq:index-polarized}, sum over $a$, and substitute \eqref{eq:l1}--\eqref{eq:l3}:
\begin{align}
\sum_{a=1}^{n+p}I_H\bigl(E_a^\top,E_a^\top\bigr)
&=\int_Me^{c}\Bigl[\bigl|S_f\bigr|^2+\sum_{i}\sigma\bigl(df(e_i)\bigr)-\bigl\langle S_f,\eta\bigr\rangle+\sum_{i}\sigma\bigl(df(e_i)\bigr)\Bigr]dv_g\notag\\
&=\int_Me^{c}\widetilde{\mathscr Q}_N(f)\,dv_g ,
\end{align}
by Definition \ref{def:howard}. The left-hand side is a sum of $n+p$ real numbers, so if it is negative, at least one of them does not exceed its average.
\end{proof}

This proves Theorem \ref{thm:A}.

\subsection{The Index Bound and Rigidity}

\begin{lemma}\label{lem:lower-bound}
Suppose $K^N\le\kappa$ along $f(M)$ for some $\kappa>0$. Then for every $W\in\Gamma(f^{-1}TN)$ with $|W|\le1$ pointwise,
\[
I_H(W,W)\ \ge\ -C_0,\qquad C_0:=\kappa\int_Me^{c}\,|df_H|^2\,dv_g .
\]
\end{lemma}

\begin{proof}
As in Proposition \ref{prop:finiteness}(i), $\mathcal R_T(W,W)\le\kappa|df_H|^2|W|^2\le\kappa|df_H|^2$, and the first two terms of $I_H(W,W)$ are nonnegative.
\end{proof}

\begin{theorem}\label{thm:index}
Let $\iota:N^n\to\mathbb R^{n+p}$ be an isometric immersion, let $H$ be bracket-generating, let $f:M\to N$ be exponentially subelliptic harmonic, and assume $K^N\le\kappa$ along $f(M)$ with $\kappa>0$. If $\widetilde{\mathcal T}(f)<0$, then $C_0>0$ and
\[
\ind(f)\ \ge\ \Bigl\lceil\frac{\bigl|\widetilde{\mathcal T}(f)\bigr|}{C_0}\Bigr\rceil .
\]
\end{theorem}

\begin{proof}
First, $C_0>0$. If $C_0=0$ then $df_H\equiv0$, so $S_f=0$ and $\sigma(df(e_i))=0$ for each $i$, whence $\widetilde{\mathscr Q}_N(f)\equiv0$ and $\widetilde{\mathcal T}(f)=0$, contrary to hypothesis.

Consider the symmetric coefficient matrix $\mathbb I:=\bigl(I_H(E_a^\top,E_b^\top)\bigr)_{a,b=1}^{n+p}$ and the linear map
\[
P:\mathbb R^{n+p}\to\Gamma(f^{-1}TN),\qquad P\lambda:=\sum_{a=1}^{n+p}\lambda_aE_a^\top ,
\]
which assigns to $\lambda$ the tangential projection along $f$ of the constant ambient vector $\Lambda:=\sum_a\lambda_aE_a$. By bilinearity,
\begin{equation}\label{eq:gram-quadratic}
\lambda^{\mathsf T}\mathbb I\lambda=I_H\bigl(P\lambda,P\lambda\bigr),\qquad\lambda\in\mathbb R^{n+p},
\end{equation}
so $\mathbb I$ is the matrix, in the standard basis of $\mathbb R^{n+p}$, of the pulled-back form $I_H\circ(P\times P)$. If $|\lambda|=1$ then $|\Lambda|=1$, so $|P\lambda|\le1$ pointwise and Lemma \ref{lem:lower-bound} gives $\lambda^{\mathsf T}\mathbb I\lambda\ge-C_0$. Hence every eigenvalue of $\mathbb I$ is at least $-C_0$. On the other hand, $\tr\mathbb I=\widetilde{\mathcal T}(f)$ by \eqref{eq:master}.

Let $k$ be the number of negative eigenvalues of $\mathbb I$, counted with multiplicity. The nonnegative eigenvalues contribute nonnegatively to the trace, and each negative one is at least $-C_0$, so
\[
\widetilde{\mathcal T}(f)=\tr\mathbb I\ \ge\ -k\,C_0,\qquad\text{that is,}\qquad k\ \ge\ \frac{|\widetilde{\mathcal T}(f)|}{C_0}.
\]
Since $k$ is an integer, $k\ge\lceil|\widetilde{\mathcal T}(f)|/C_0\rceil$.

Let $\mathcal E_-\subset\mathbb R^{n+p}$ be spanned by the eigenvectors belonging to negative eigenvalues, so that $\dim\mathcal E_-=k$ and $\lambda^{\mathsf T}\mathbb I\lambda<0$ for $\lambda\in\mathcal E_-\setminus\{0\}$. The restriction $P|_{\mathcal E_-}$ is injective. If $\lambda\in\mathcal E_-$ and $P\lambda=0$, then $\lambda^{\mathsf T}\mathbb I\lambda=I_H(0,0)=0$ by \eqref{eq:gram-quadratic}, forcing $\lambda=0$. Lemma \ref{lem:test-space}, applied to $V=\mathcal E_-$ and $\Phi=P|_{\mathcal E_-}$, gives $\ind(f)\ge k$.
\end{proof}

This proves Theorem \ref{thm:B}.

\begin{theorem}\label{thm:rigidity}
Let $\iota$, $H$ and $f$ be as in Theorem \ref{thm:index}, and assume
\begin{equation}\label{eq:rigidity-hyp}
\widetilde{\mathscr Q}_N(f)\le0\ \text{on }M,
\qquad\text{and}\qquad
\widetilde{\mathscr Q}_N(f)(x)=0,
\end{equation}
only where $df_H(x)=0$. If $f$ is stable, then $f$ is constant.
\end{theorem}

\begin{proof}
Stability gives $\sum_aI_H(E_a^\top,E_a^\top)\ge0$. By \eqref{eq:master} this sum equals $\widetilde{\mathcal T}(f)$, which is $\le0$ by the first hypothesis. Hence $\widetilde{\mathcal T}(f)=0$. The integrand $e^{c}\widetilde{\mathscr Q}_N(f)$ is continuous, nonpositive and of vanishing integral, so it vanishes identically and $\widetilde{\mathscr Q}_N(f)\equiv0$. The second hypothesis gives $df_H\equiv0$, and Lemma \ref{lem:horiz-const} shows that $f$ is constant.
\end{proof}

\section{The Parallel Test Space}\label{sec:par}

We turn to the second test space. It may be trivial, but when it is not, the index form on it is computed rather than estimated.

\begin{definition}\label{def:Pf}
Let
\[
P(f):=\bigl\{W\in\Gamma(f^{-1}TN):\ \tilde\nabla_YW=0\ \text{for all }Y\in\Gamma(TM)\bigr\}
\]
be the space of parallel sections of the pullback bundle, and put $k(f):=\dim P(f)$.
\end{definition}

A parallel section is determined by its value at any point $x_0$, and a vector of $T_{f(x_0)}N$ extends to a parallel section exactly when it is invariant under parallel transport along all loops at $x_0$. Thus evaluation at $x_0$ identifies $P(f)$ with the subspace of $T_{f(x_0)}N$ fixed by the holonomy group of the pullback connection. In particular $k(f)\le n$, with equality exactly when $f^{-1}TN$ admits a global parallel orthonormal frame, that is, when the pullback holonomy is trivial. If the holonomy group has no nonzero fixed vector, then $P(f)=\{0\}$ and the parallel test space is trivial. Note that $k(f)=n$ is stronger than the topological triviality of $f^{-1}TN$: for instance, the identity map of a non-flat $2$-torus has trivial tangent bundle but no global parallel frame. Since the connection is metric, inner products of parallel sections are constant, so $P(f)$ admits a basis $W_1,\dots,W_k$ that is orthonormal at every point. For each $\alpha$ put
\begin{equation}\label{eq:xi-rho}
\xi_\alpha:=\sum_{i=1}^m\bigl\langle W_\alpha,df(e_i)\bigr\rangle e_i\in\Gamma(H),
\qquad
\rho_{\alpha\beta}:=\mathcal R_T(W_\alpha,W_\beta)\in C^\infty(M).
\end{equation}
Here $\xi_\alpha$ is the $g_H$-dual of the one-form $X\mapsto\langle W_\alpha,df(X)\rangle$ on $H$, hence independent of the frame, and $(\rho_{\alpha\beta})$ is a symmetric matrix of functions by the symmetry of $\mathcal R_T$.

\begin{theorem}[Reduction]\label{thm:reduction}
Let $f$ be exponentially subelliptic harmonic, $M$ closed, $k=k(f)\ge1$, and let
\[
\Phi:C^\infty(M;\mathbb R^k)\to\Gamma(f^{-1}TN),\qquad \Phi(v):=\sum_{\alpha=1}^kv_\alpha W_\alpha .
\]
Then $\Phi$ is injective and $I_H(\Phi v,\Phi v)=\mathcal Q_f(v)$, where
\begin{equation}\label{eq:reduced}
\mathcal Q_f(v):=\int_Me^{c}\Bigl[\Bigl(\sum_{\alpha}\bigl\langle\nabla^Hv_\alpha,\xi_\alpha\bigr\rangle\Bigr)^2+\sum_{\alpha}\bigl|\nabla^Hv_\alpha\bigr|^2-\sum_{\alpha,\beta}\rho_{\alpha\beta}v_\alpha v_\beta\Bigr]dv_g .
\end{equation}
Hence $\ind(f)\ge\ind(\mathcal Q_f)$. If $k(f)=n$, then $\Phi$ is bijective, $\ind(f)=\ind(\mathcal Q_f)$, and $\nul(f)$ equals the dimension of the radical of $\mathcal Q_f$ on the completion $\mathcal H^1_H(M;\mathbb R^n)$ of $C^\infty(M;\mathbb R^n)$ in the norm $\bigl(\int_Me^{c}(\sum_\alpha|\nabla^Hv_\alpha|^2+|v|^2)dv_g\bigr)^{1/2}$.
\end{theorem}

\begin{proof}
Injectivity follows from pointwise orthonormality of the $W_\alpha$: $|\Phi v|^2=|v|^2$. Since $\tilde\nabla W_\alpha=0$,
\[
\tilde\nabla_{e_i}\Phi v=\sum_{\alpha=1}^k\bigl(e_iv_\alpha\bigr)W_\alpha,
\qquad\text{hence}\qquad
\bigl|\nabla_H\Phi v\bigr|^2=\sum_{\alpha=1}^k\bigl|\nabla^Hv_\alpha\bigr|^2 ,
\]
again by orthonormality. Moreover
\[
\mathcal F(\Phi v)=\sum_{i=1}^m\sum_{\alpha=1}^k\bigl(e_iv_\alpha\bigr)\bigl\langle W_\alpha,df(e_i)\bigr\rangle=\sum_{\alpha=1}^k\bigl\langle\nabla^Hv_\alpha,\xi_\alpha\bigr\rangle ,
\]
and bilinearity gives $\mathcal R_T(\Phi v,\Phi v)=\sum_{\alpha,\beta}\rho_{\alpha\beta}v_\alpha v_\beta$. Substituting into \eqref{eq:index-polarized} gives \eqref{eq:reduced}, and Lemma \ref{lem:test-space} gives $\ind(f)\ge\ind(\mathcal Q_f)$.

If $k(f)=n$, the $W_\alpha$ form a global orthonormal frame of $f^{-1}TN$, so every section is of the form $\Phi v$ and $\Phi$ is bijective. Lemma \ref{lem:test-space} gives $\ind(f)=\ind(\mathcal Q_f)$. By the two identities above, $\Phi$ is an isometry for the norms $\|\cdot\|_{\mathcal H^1_H}$. It therefore extends to an isometric isomorphism $\mathcal H^1_H(M;\mathbb R^n)\to\mathcal H^1_H$ which, by continuity and Proposition \ref{prop:finiteness}(i), carries the continuous extension of $\mathcal Q_f$ to that of $I_H$. The radicals correspond, which proves the statement on the nullity.
\end{proof}

This proves Theorem \ref{thm:C}. When $k>1$, the first term of \eqref{eq:reduced} contains the second-order cross terms $2\langle\nabla^Hv_\alpha,\xi_\alpha\rangle\langle\nabla^Hv_\beta,\xi_\beta\rangle$, so $\mathcal Q_f$ is a coupled form on vector-valued functions and in general does not split into scalar forms. If $\xi_\alpha=0$ for all but at most one index and $(\rho_{\alpha\beta})$ is diagonal, the form decouples into scalar blocks. All blocks with $\xi_\alpha=0$ are then of Schr\"odinger type with respect to the weighted horizontal Laplacian, while a block with $\xi_\alpha\ne0$ contains the additional term $\langle\nabla^Hv_\alpha,\xi_\alpha\rangle^2$, which modifies its principal part. This is the situation of Section \ref{sec:heis}. Restricting to constant $v$ gives a finite-dimensional criterion.

\begin{corollary}\label{cor:J}
Let $\mathbb J_{\alpha\beta}:=I_H(W_\alpha,W_\beta)=-\int_Me^{c}\rho_{\alpha\beta}\,dv_g$. Then $\ind(f)$ is at least the number of negative eigenvalues of the symmetric $k\times k$ matrix $\mathbb J$.
\end{corollary}

\begin{proof}
For constant $v=\lambda\in\mathbb R^k$ one has $\nabla^Hv_\alpha=0$, so \eqref{eq:reduced} gives $\mathcal Q_f(\lambda)=\lambda^{\mathsf T}\mathbb J\lambda$. Applying Lemma \ref{lem:test-space} to $\Phi$ restricted to the span of the negative eigenvectors of $\mathbb J$ gives the claim.
\end{proof}

\begin{corollary}\label{cor:positive-curv}
Assume that $H$ is bracket-generating, that $K^N>0$ along $f(M)$, and that for every nonzero $W\in P(f)$ there is a point $x\in M$ with $\image df_H(x)\not\subset\mathbb R\,W(x)$. Then $\mathbb J$ is negative definite and $\ind(f)\ge k(f)$. The last hypothesis holds if $f$ is nonconstant and every $W\in P(f)$ is pointwise orthogonal to $\image df_H$.
\end{corollary}

\begin{proof}
Let $W\in P(f)\setminus\{0\}$. Since $W$ is parallel and nonzero, it vanishes nowhere. Fix $i$ and a point. If $df(e_i)$ and $W$ are linearly independent there, they span a plane $\pi$ and
\[
\tilde R_N\bigl(df(e_i),W,W,df(e_i)\bigr)=K^N(\pi)\Bigl(|df(e_i)|^2|W|^2-\langle df(e_i),W\rangle^2\Bigr)>0 .
\]
If they are linearly dependent, then $\tilde R_N(df(e_i),W,W,df(e_i))=0$ by the skew-symmetry of the curvature tensor. Hence $\mathcal R_T(W,W)\ge0$, and $\mathcal R_T(W,W)(x)=0$ only if every $df(e_i)(x)$ lies in $\mathbb RW(x)$, that is, only if $\image df_H(x)\subset\mathbb RW(x)$. By hypothesis this fails at some point, hence on a nonempty open set, so
\[
I_H(W,W)=-\int_Me^{c}\mathcal R_T(W,W)\,dv_g<0 .
\]
Thus $\mathbb J$ is negative definite, and Corollary \ref{cor:J} gives $\ind(f)\ge k(f)$. If $W\perp\image df_H$ pointwise, the inclusion $\image df_H(x)\subset\mathbb RW(x)$ forces $df_H(x)=0$, and by Lemma \ref{lem:horiz-const} a nonconstant $f$ has a point where this fails.
\end{proof}

\begin{remark}\label{rem:parallel-limits}
If $K^N\le0$ then $\rho\le0$ as a matrix, so $\mathcal Q_f\ge0$ and the parallel test space gives nothing, in accordance with the stability statement following Definition \ref{def:index-form}. The method is useful when nonzero parallel sections exist and the resulting reduced form has negative directions. Maps whose images lie in totally geodesic subspheres provide an important class of examples. See Proposition \ref{prop:subsphere}.
\end{remark}

\section{Applications to Target Manifolds}\label{sec:targets}

\subsection{Round Spheres}

Throughout this subsection $N=S^n(r)\subset\mathbb R^{n+1}$ carries the inward unit normal $\nu=-x/r$, so that
\[
B(X,Y)=\tfrac1r\langle X,Y\rangle\nu,\qquad
\mathcal A^\nu=\tfrac1r\operatorname{Id},\qquad
\eta=\tfrac nr\nu,\qquad
\sigma(X)=\tfrac1{r^2}|X|^2,\qquad
K^N\equiv\tfrac1{r^2}.
\]

\begin{theorem}\label{thm:sphere}
For every smooth $f:M\to S^n(r)$,
\begin{equation}\label{eq:Qsphere}
\widetilde{\mathscr Q}_{S^n(r)}(f)=\frac{1}{r^2}\,|df_H|^2\Bigl(|df_H|^2-(n-2)\Bigr).
\end{equation}
Let $n\ge3$, let $H$ be bracket-generating, and let $f$ be exponentially subelliptic harmonic with $|df_H|^2<n-2$ on $M$. Then $f$ is unstable if it is nonconstant, and constant if it is stable.
\end{theorem}

\begin{proof}
From $B(X,Y)=\frac1r\langle X,Y\rangle\nu$ we get $S_f=\frac1r|df_H|^2\nu$ and $\sum_i\sigma(df(e_i))=\frac1{r^2}|df_H|^2$. Substituting into \eqref{eq:Qdef},
\[
\widetilde{\mathscr Q}_{S^n(r)}(f)
=\frac{|df_H|^4}{r^2}-\frac{n|df_H|^2}{r^2}+\frac{2|df_H|^2}{r^2}
=\frac{|df_H|^2}{r^2}\Bigl(|df_H|^2-(n-2)\Bigr).
\]
Assume now $n\ge3$ and $|df_H|^2<n-2$. Then $\widetilde{\mathscr Q}_{S^n(r)}(f)\le0$, with strict inequality precisely where $df_H\ne0$. If $f$ is nonconstant, Lemma \ref{lem:horiz-const} shows that $\{df_H\ne0\}$ is nonempty and open, so $\widetilde{\mathcal T}(f)<0$ and $f$ is unstable by Theorem \ref{thm:master}.

For rigidity we verify \eqref{eq:rigidity-hyp}. The first condition has just been shown. Viewed as a function of $|df_H|^2$, the right-hand side of \eqref{eq:Qsphere} has the roots $0$ and $n-2$, and the latter is excluded by the strict hypothesis. Hence $\widetilde{\mathscr Q}_{S^n(r)}(f)(x)=0$ forces $df_H(x)=0$. Theorem \ref{thm:rigidity} applies.
\end{proof}

\begin{remark}\label{rem:leung}
For $n\le2$ the condition $|df_H|^2<n-2$ admits no solution with $df_H\ne0$, so the criteria are nontrivial only for $n\ge3$. The identity \eqref{eq:Qsphere} holds for all $n$. It is instructive to compare with the Dirichlet case. At a critical point of the horizontal Dirichlet energy $E_D(f)=\frac12\int_M|df_H|^2\,dv_g$, the index form is obtained from \eqref{eq:index-polarized} by replacing the weight $e^{c}$ by $1$ and omitting the exponential coupling term $\mathcal F(W)\mathcal F(W')$. The sum over the projected frame on $S^n(r)$ then has integrand
\[
2\sum_i\sigma\bigl(df(e_i)\bigr)-\langle S_f,\eta\rangle=-\frac{n-2}{r^2}|df_H|^2 ,
\]
negative for $n\ge3$ with no smallness condition. This is the mechanism behind Leung's theorem that nonconstant harmonic maps into $S^n$, $n\ge3$, are unstable \cite{Leu82}. Its counterpart for maps from $S^n$ is due to Xin \cite{Xin80}. The quartic term $|S_f|^2=r^{-2}|df_H|^4$ produced by the exponential factor competes with it and makes the averaged trace positive once $|df_H|^2>n-2$ everywhere. The pointwise smallness condition therefore guarantees negativity of the averaged trace. It is sufficient but not necessary for instability, as the examples of Section \ref{sec:heis} show.
\end{remark}

\begin{corollary}\label{cor:index-sphere}
Let $n\ge3$, let $H$ be bracket-generating, and let $f:M\to S^n(r)$ be a nonconstant exponentially subelliptic harmonic map. Write $s:=|df_H|^2$ and
\begin{equation}\label{eq:weighted-mean}
\bar s_w:=\frac{\displaystyle\int_Me^{s/2}s^2\,dv_g}{\displaystyle\int_Me^{s/2}s\,dv_g}\,.
\end{equation}
With $\kappa=r^{-2}$ in Theorem \ref{thm:index},
\begin{equation}\label{eq:sphere-ratio}
\frac{\widetilde{\mathcal T}(f)}{C_0}=\bar s_w-(n-2),
\end{equation}
and consequently
\[
\ind(f)\ \ge\ \max\bigl\{0,\ \lceil(n-2)-\bar s_w\rceil\bigr\}\ \ge\ \max\bigl\{0,\ \lceil(n-2)-\textstyle\sup_Ms\rceil\bigr\}.
\]
\end{corollary}

\begin{proof}
Since $f$ is nonconstant, Lemma \ref{lem:horiz-const} gives $\int_Me^{s/2}s\,dv_g>0$, so $\bar s_w$ is well defined and positive. Moreover $s^2\le(\sup_Ms)\,s$ pointwise, whence $\bar s_w\le\sup_Ms$. Here $c=s/2$, so by \eqref{eq:Qsphere} and Lemma \ref{lem:lower-bound},
\begin{align}
\widetilde{\mathcal T}(f)&=\frac1{r^2}\Bigl[\int_Me^{s/2}s^2\,dv_g-(n-2)\int_Me^{s/2}s\,dv_g\Bigr],\\
C_0&=\frac1{r^2}\int_Me^{s/2}s\,dv_g>0 ,
\end{align}
and division gives \eqref{eq:sphere-ratio}. If $\bar s_w\ge n-2$, both asserted bounds are trivial. If $\bar s_w<n-2$, then $\widetilde{\mathcal T}(f)<0$ by \eqref{eq:sphere-ratio}, and Theorem \ref{thm:index} gives the first bound. The second follows from $\bar s_w\le\sup_Ms$.
\end{proof}

\begin{remark}
The bound of Corollary \ref{cor:index-sphere} is nontrivial exactly when $\bar s_w<n-2$, which by \eqref{eq:sphere-ratio} is precisely the condition $\widetilde{\mathcal T}(f)<0$ of Theorem \ref{thm:index}. For spherical targets the trace-to-curvature quotient is evaluated exactly, without replacing the weighted mean by a pointwise supremum, and the radius cancels. The resulting index bound need not be sharp: it uses a single finite-dimensional test family, the uniform lower bound $-C_0$ for every negative eigenvalue, and the trace in place of the full negative spectrum. For instance, for $n=3$ and the solutions of Section \ref{sec:heis} with $\Sigma=\frac12$ it gives $\ind(f)\ge1$, whereas Theorem \ref{thm:heis-spectral} gives $\ind(f)\ge2$. The two sources of loss are separated in Remark \ref{rem:ext-matrix}. The weight $e^{s/2}s\,dv_g$ favours the region where $s$ is large: by the Cauchy--Schwarz inequality for the measure $e^{s/2}dv_g$,
\[
\frac{\int_Me^{s/2}s\,dv_g}{\int_Me^{s/2}dv_g}\ \le\ \bar s_w\ \le\ \sup_Ms .
\]
The index bound is an integral condition, whereas the rigidity argument of Theorem \ref{thm:sphere} uses the sign condition at every point. Finally, for the solutions of Section \ref{sec:heis} the density $s$ is constant, so there $\bar s_w=\sup_Ms$ and the weighted bound brings no improvement.
\end{remark}

The parallel test space applies to spheres when the image degenerates.

\begin{proposition}\label{prop:subsphere}
Let $P\subset\mathbb R^{n+1}$ be a linear subspace with $\dim P=k+1$, $1\le k\le n-1$, and put $S^k(r):=S^n(r)\cap P$, a totally geodesic subsphere. Let $H$ be bracket-generating and let $f:M\to S^n(r)$ be a nonconstant exponentially subelliptic harmonic map with $f(M)\subset S^k(r)$. For $w\in P^\perp$, the section $W_w$ of $f^{-1}TS^n(r)$ with constant value $w$ lies in $P(f)$, and
\begin{equation}\label{eq:subsphere-index}
I_H\bigl(W_w,W_{w'}\bigr)=-\frac{\langle w,w'\rangle}{r^2}\int_Me^{c}\,|df_H|^2\,dv_g
\qquad(w,w'\in P^\perp).
\end{equation}
Consequently $k(f)\ge n-k$, and the index of $f$ as a map into $S^n(r)$ satisfies $\ind(f)\ge n-k$.
\end{proposition}

\begin{proof}
For $x\in M$ we have $f(x)\in P$, so $w\perp f(x)$ for $w\in P^\perp$, that is, $w\in T_{f(x)}S^n(r)=f(x)^\perp$. Thus $W_w$ is a section of $f^{-1}TS^n(r)$. Let $D$ be the flat connection of $\mathbb R^{n+1}$. Since $w$ is constant, $D_YW_w=0$ for every $Y\in\Gamma(TM)$. The pullback connection is the tangential projection of $D$, so $\tilde\nabla_YW_w=(D_YW_w)^\top=0$ and $W_w\in P(f)$. As $\dim P^\perp=n-k$, we get $k(f)\ge n-k$.

Since $f(M)\subset P$, each $df(e_i)$ lies in $P$, so $\langle df(e_i),w\rangle=0$ for $w\in P^\perp$. The curvature of $S^n(r)$ is $\tilde R^N(X,Y)Z=r^{-2}(\langle Y,Z\rangle X-\langle X,Z\rangle Y)$, hence
\begin{align}
\tilde R_N\bigl(df(e_i),W_w,W_{w'},df(e_i)\bigr)
&=\frac{1}{r^2}\Bigl(\langle w,w'\rangle|df(e_i)|^2-\langle df(e_i),w'\rangle\langle w,df(e_i)\rangle\Bigr)\notag\\
&=\frac{\langle w,w'\rangle}{r^2}|df(e_i)|^2 .
\end{align}
Summing over $i$ and using $I_H(W_w,W_{w'})=-\int_Me^{c}\mathcal R_T(W_w,W_{w'})\,dv_g$, valid since both sections are parallel, gives \eqref{eq:subsphere-index}. Since $f$ is nonconstant, $\int_Me^{c}|df_H|^2dv_g>0$ by Lemma \ref{lem:horiz-const}, so the right-hand side of \eqref{eq:subsphere-index} is a negative multiple of the inner product on $P^\perp$. Corollary \ref{cor:J}, applied with a basis of $P(f)$ extending an orthonormal basis of $\{W_w:w\in P^\perp\}$, gives $\ind(f)\ge n-k$.
\end{proof}

\subsection{Products of Spheres}

For products of spheres and for convex hypersurfaces only the extrinsic test space is used.

\begin{theorem}\label{thm:products}
Let $n_1,n_2\ge3$, $n:=n_1+n_2$, $r_1,r_2>0$, and let
\[
N=S^{n_1}(r_1)\times S^{n_2}(r_2)\subset\mathbb R^{n_1+1}\times\mathbb R^{n_2+1}=\mathbb R^{n+2}
\]
carry the product embedding, with inward unit normals
\[
\nu_1=\Bigl(-\frac{x_1}{r_1},\,0\Bigr),\qquad \nu_2=\Bigl(0,\,-\frac{x_2}{r_2}\Bigr).
\]
Let $f:M\to N$ be smooth, write $df(e_i)=X_{1,i}+X_{2,i}$ according to the product splitting, and put $s_\ell:=\sum_{i=1}^m|X_{\ell,i}|^2$, so that $|df_H|^2=s_1+s_2$. Then
\begin{equation}\label{eq:product-exact}
\widetilde{\mathscr Q}_N(f)=\frac{s_1\bigl(s_1-(n_1-2)\bigr)}{r_1^2}+\frac{s_2\bigl(s_2-(n_2-2)\bigr)}{r_2^2}.
\end{equation}
If $H$ is bracket-generating and $f$ is exponentially subelliptic harmonic with $s_1<n_1-2$ and $s_2<n_2-2$ on $M$, then $f$ is unstable if it is nonconstant, and constant if it is stable.
\end{theorem}

\begin{proof}
For $X=X_1+X_2$ one has $B(X,X)=\frac1{r_1}|X_1|^2\nu_1+\frac1{r_2}|X_2|^2\nu_2$, hence
\[
S_f=\frac{s_1}{r_1}\nu_1+\frac{s_2}{r_2}\nu_2,
\qquad
\eta=\frac{n_1}{r_1}\nu_1+\frac{n_2}{r_2}\nu_2 .
\]
The normals $\nu_1,\nu_2$ are orthonormal, and $\mathcal A^{\nu_\ell}$ acts as $r_\ell^{-1}\operatorname{Id}$ on the $\ell$-th factor and as $0$ on the other. Therefore
\[
|S_f|^2=\frac{s_1^2}{r_1^2}+\frac{s_2^2}{r_2^2},
\qquad
\langle S_f,\eta\rangle=\frac{n_1s_1}{r_1^2}+\frac{n_2s_2}{r_2^2},
\qquad
\sum_{i=1}^m\sigma\bigl(df(e_i)\bigr)=\frac{s_1}{r_1^2}+\frac{s_2}{r_2^2},
\]
and substitution into \eqref{eq:Qdef} gives \eqref{eq:product-exact}.

Under the stated hypotheses, each summand of \eqref{eq:product-exact} is nonpositive and vanishes if and only if the corresponding $s_\ell$ vanishes. If $f$ is nonconstant, then $s_1+s_2>0$ on a nonempty open set by Lemma \ref{lem:horiz-const}, so $\widetilde{\mathcal T}(f)<0$ and $f$ is unstable by Theorem \ref{thm:master}. For rigidity, the first condition of \eqref{eq:rigidity-hyp} holds, and $\widetilde{\mathscr Q}_N(f)(x)=0$ forces $s_1(x)=s_2(x)=0$, that is, $df_H(x)=0$. Theorem \ref{thm:rigidity} applies.
\end{proof}

\begin{corollary}\label{cor:index-product}
In the notation of Theorem \ref{thm:products}, let $H$ be bracket-generating and let $f$ be a nonconstant exponentially subelliptic harmonic map. Put
\[
\Sigma_\ell:=\sup_Ms_\ell,\qquad
\delta_\ell:=(n_\ell-2)-\Sigma_\ell\quad(\ell=1,2),\qquad
\kappa:=\max\{r_1^{-2},r_2^{-2}\}.
\]
If $\delta_1>0$ and $\delta_2>0$, then
\[
\ind(f)\ \ge\ \Bigl\lceil\frac{\min\{\delta_1r_1^{-2},\ \delta_2r_2^{-2}\}}{\kappa}\Bigr\rceil .
\]
\end{corollary}

\begin{proof}
We first bound the sectional curvature. Let $U=(U_1,U_2)$ and $V=(V_1,V_2)$ be orthonormal in $T_{(q_1,q_2)}N$. Since the curvature tensor of a product metric splits,
\[
K^N(U\wedge V)=\frac{|U_1\wedge V_1|^2}{r_1^2}+\frac{|U_2\wedge V_2|^2}{r_2^2},
\qquad |X\wedge Y|^2:=|X|^2|Y|^2-\langle X,Y\rangle^2 .
\]
Expanding $|U\wedge V|^2$,
\begin{align}
|U\wedge V|^2-|U_1\wedge V_1|^2-|U_2\wedge V_2|^2
&=|U_1|^2|V_2|^2+|U_2|^2|V_1|^2-2\langle U_1,V_1\rangle\langle U_2,V_2\rangle\notag\\
&\ge|U_1|^2|V_2|^2+|U_2|^2|V_1|^2-2|U_1||V_1||U_2||V_2|\notag\\
&=\bigl(|U_1||V_2|-|U_2||V_1|\bigr)^2\ \ge\ 0 .
\end{align}
Since $|U\wedge V|^2=1$, we obtain $|U_1\wedge V_1|^2+|U_2\wedge V_2|^2\le1$, whence $K^N\le\kappa$.

By \eqref{eq:product-exact} and $s_\ell\le\Sigma_\ell$,
\begin{align}
\widetilde{\mathscr Q}_N(f)
&\le-\frac{\delta_1s_1}{r_1^2}-\frac{\delta_2s_2}{r_2^2}\notag\\
&\le-\min\{\delta_1r_1^{-2},\delta_2r_2^{-2}\}\,(s_1+s_2)
=-\min\{\delta_1r_1^{-2},\delta_2r_2^{-2}\}\,|df_H|^2 .
\end{align}
Since $f$ is nonconstant, $\widetilde{\mathcal T}(f)<0$ and
\[
|\widetilde{\mathcal T}(f)|\ge\min\{\delta_1r_1^{-2},\delta_2r_2^{-2}\}\int_Me^{c}|df_H|^2\,dv_g,
\qquad
C_0=\kappa\int_Me^{c}|df_H|^2\,dv_g .
\]
Theorem \ref{thm:index} concludes.
\end{proof}

\begin{remark}
The bound $K^N\le\kappa$ is attained on planes tangent to the factor of smaller radius, but it is crude for most planes, since a product contains mixed two-planes of vanishing curvature. It is nevertheless the only information used in Lemma \ref{lem:lower-bound}, which is applied to arbitrary test fields. For simplicity, Corollary \ref{cor:index-product} is stated in terms of the pointwise suprema $\Sigma_\ell$. An integral refinement can also be obtained from the exact trace identity. With $s=s_1+s_2$, $D=\int_Me^{s/2}s\,dv_g>0$, $A_\ell=\int_Me^{s/2}s_\ell\,dv_g$, $p_\ell=A_\ell/D$ and, for $A_\ell>0$, $\bar s_{\ell,w}=A_\ell^{-1}\int_Me^{s/2}s_\ell^2\,dv_g$, formula \eqref{eq:product-exact} gives
\[
\frac{\widetilde{\mathcal T}(f)}{C_0}=\sum_{\ell:\,A_\ell>0}\frac{p_\ell}{\kappa r_\ell^2}\bigl(\bar s_{\ell,w}-(n_\ell-2)\bigr),
\]
and hence $\ind(f)\ge\max\bigl\{0,\bigl\lceil\sum_{\ell:A_\ell>0}\frac{p_\ell}{\kappa r_\ell^2}\bigl((n_\ell-2)-\bar s_{\ell,w}\bigr)\bigr\rceil\bigr\}$ by Theorem \ref{thm:index}.
\end{remark}

\subsection{Convex Hypersurfaces}

Let $N^n$ be a two-sided immersed hypersurface in $\mathbb R^{n+1}$ with global unit normal $\nu$, let $\mathcal A:=\mathcal A^\nu$ be the shape operator, denote the principal curvatures by $\kappa_1\le\cdots\le\kappa_n$, and put $h:=\tr\mathcal A$. Then
\[
B(X,Y)=\bigl\langle\mathcal AX,Y\bigr\rangle\nu,\qquad\eta=h\nu,\qquad\sigma(X)=\bigl|\mathcal AX\bigr|^2 .
\]

\begin{lemma}\label{lem:spectral-prep}
Assume $\mathcal A\ge0$ at the point $f(x)$. Expand $df(e_i)=\sum_{A=1}^nc_{iA}\varepsilon_A$ in an orthonormal eigenbasis $\{\varepsilon_A\}$ of $\mathcal A$ at $f(x)$, with $\mathcal A\varepsilon_A=\kappa_A\varepsilon_A$, and put
\[
\mu_A:=\sum_{i=1}^mc_{iA}^2\ \ge0,
\qquad
\theta:=\sum_{i=1}^m\bigl\langle\mathcal A(df(e_i)),df(e_i)\bigr\rangle=\sum_{A=1}^n\kappa_A\mu_A .
\]
Then $\sum_A\mu_A=|df_H|^2$ and
\begin{align}
0\ \le\ \theta\ &\le\ \kappa_n|df_H|^2,\label{eq:theta-bound}\\
\kappa_n\theta-\sum_{A=1}^n\kappa_A^2\mu_A&=\sum_{A=1}^n\kappa_A\bigl(\kappa_n-\kappa_A\bigr)\mu_A\ \ge\ 0,\label{eq:defect}\\
\widetilde{\mathscr Q}_N(f)=\theta^2-\theta h+2\sum_{A=1}^n\kappa_A^2\mu_A&\ \le\ \theta\bigl(\theta-h+2\kappa_n\bigr).\label{eq:l4.2}
\end{align}
\end{lemma}

\begin{proof}
Since $B(X,X)=\langle\mathcal AX,X\rangle\nu$, we have $S_f=\theta\nu$, so $|S_f|^2=\theta^2$ and $\langle S_f,\eta\rangle=\theta h$. Moreover
\[
\sum_{i=1}^m\sigma\bigl(df(e_i)\bigr)=\sum_{i=1}^m\bigl|\mathcal A(df(e_i))\bigr|^2=\sum_{A=1}^n\kappa_A^2\mu_A ,
\]
which gives the equality in \eqref{eq:l4.2}. As $0\le\kappa_A\le\kappa_n$, summing $\kappa_A\mu_A\le\kappa_n\mu_A$ gives \eqref{eq:theta-bound}, and \eqref{eq:defect} is an identity whose right-hand side is a sum of nonnegative terms. Substituting $\sum_A\kappa_A^2\mu_A\le\kappa_n\theta$ into the equality of \eqref{eq:l4.2} gives the inequality.
\end{proof}

\begin{theorem}\label{thm:spectral}
Let $\mathcal A\ge0$ and suppose that at a point of $M$ one has $\kappa_n>0$ and
\begin{equation}\label{eq:threshold}
|df_H|^2<\frac{h}{\kappa_n}-2 .
\end{equation}
Then $\widetilde{\mathscr Q}_N(f)\le0$ at that point, with strict inequality if $\theta>0$. If $\mathcal A>0$, then $\theta>0$ whenever $df_H\ne0$.
\end{theorem}

\begin{proof}
By \eqref{eq:theta-bound} and \eqref{eq:threshold},
\[
\theta-h+2\kappa_n\ \le\ \kappa_n|df_H|^2-h+2\kappa_n
=\kappa_n\Bigl(|df_H|^2-\frac{h}{\kappa_n}+2\Bigr)\ <\ 0 .
\]
As $\theta\ge0$, \eqref{eq:l4.2} gives $\widetilde{\mathscr Q}_N(f)\le\theta(\theta-h+2\kappa_n)\le0$, with strict inequality when $\theta>0$. This covers both the case $df_H=0$, where $\theta=0$, and the case $df_H\ne0$. Finally, $\mathcal A>0$ gives $\ker\mathcal A=0$, so $df_H\ne0$ forces $\theta\ge\kappa_1|df_H|^2>0$.
\end{proof}

\begin{remark}\label{rem:anisotropic}
By \eqref{eq:defect}, equality holds in $\sum_A\kappa_A^2\mu_A\le\kappa_n\theta$ if and only if $\mu_A=0$ for every $A$ with $0<\kappa_A<\kappa_n$, that is, if and only if
\[
\image\bigl(df_H\bigr)\subseteq\ker\mathcal A+E_{\kappa_n},
\]
where $E_{\kappa_n}$ is the eigenspace of $\mathcal A$ for $\kappa_n$. When $\mathcal A>0$ this means that $df_H$ takes values in $E_{\kappa_n}$. When $\kappa_n>0$ one has $h/\kappa_n\in[1,n]$, so the threshold in \eqref{eq:threshold} lies in $[-1,n-2]$ and has to be evaluated pointwise.
\end{remark}

\begin{theorem}\label{thm:convex-full}
Let $\mathcal A>0$, let $H$ be bracket-generating, and let $f:M\to N$ be exponentially subelliptic harmonic, satisfying \eqref{eq:threshold} at every point of $M$. Then:
\begin{enumerate}
\item[\rm(i)] if $f$ is nonconstant, then $f$ is unstable;
\item[\rm(ii)] if $f$ is stable, then $f$ is constant;
\item[\rm(iii)] suppose $f$ is nonconstant, and put
\[
\Theta(q):=\frac{h(q)}{\kappa_n(q)}-2,\quad
\Theta_{\min}:=\min_{f(M)}\Theta,\quad
\kappa_{\min}:=\min_{f(M)}\kappa_1,\quad
\kappa_{\max}:=\max_{f(M)}\kappa_n,\quad
\Sigma:=\sup_M|df_H|^2 .
\]
If $\Sigma<\Theta_{\min}$, then $\alpha_0:=\min_{q\in f(M)}\kappa_1(q)\kappa_n(q)\bigl(\Theta(q)-\Sigma\bigr)$ satisfies $\alpha_0\ge\kappa_{\min}^2(\Theta_{\min}-\Sigma)>0$, and
\[
\ind(f)\ \ge\ \Bigl\lceil\frac{\alpha_0}{\kappa_{\max}^2}\Bigr\rceil
\ \ge\ \Bigl\lceil\frac{\kappa_{\min}^2}{\kappa_{\max}^2}\bigl(\Theta_{\min}-\Sigma\bigr)\Bigr\rceil .
\]
\end{enumerate}
\end{theorem}

\begin{proof}
By Theorem \ref{thm:spectral}, $\widetilde{\mathscr Q}_N(f)\le0$ on $M$, with strict inequality wherever $df_H\ne0$. If $f$ is nonconstant, then $\widetilde{\mathcal T}(f)<0$ by Lemma \ref{lem:horiz-const}, and Theorem \ref{thm:master} gives (i). For (ii), the first condition of \eqref{eq:rigidity-hyp} holds; moreover $\widetilde{\mathscr Q}_N(f)(x)=0$ forces $\theta(x)=0$, hence $df_H(x)=0$ because $\ker\mathcal A=0$. Theorem \ref{thm:rigidity} applies.

For (iii), $f(M)$ is compact and the principal curvatures are continuous, so the minima and the maximum are attained. Since $\mathcal A>0$, we have $\kappa_1,\kappa_n>0$ on $f(M)$, and $\Theta-\Sigma\ge\Theta_{\min}-\Sigma>0$ there by the assumption $\Sigma<\Theta_{\min}$. Hence
\[
\alpha_0\ \ge\ \kappa_{\min}\cdot\kappa_{\min}\cdot\bigl(\Theta_{\min}-\Sigma\bigr)\ >\ 0 ,
\]
using $\kappa_n\ge\kappa_1\ge\kappa_{\min}$. This also gives the second inequality of the statement. By the Gauss equation \eqref{eq:gauss}, for orthonormal $U,V\in T_qN$,
\begin{align}
K^N(U\wedge V)&=\langle\mathcal AU,U\rangle\langle\mathcal AV,V\rangle-\langle\mathcal AU,V\rangle^2\notag\\
&\le\langle\mathcal AU,U\rangle\langle\mathcal AV,V\rangle\ \le\ \kappa_n(q)^2\ \le\ \kappa_{\max}^2 ,
\end{align}
so Lemma \ref{lem:lower-bound} applies with $\kappa=\kappa_{\max}^2$. Write $-h+2\kappa_n=-\kappa_n\Theta$. Using \eqref{eq:l4.2} together with $\theta\le\kappa_n|df_H|^2\le\kappa_n\Sigma$ and $\theta\ge\kappa_1|df_H|^2$,
\begin{align}
\widetilde{\mathscr Q}_N(f)
&\le\theta\bigl(\theta-\kappa_n\Theta\bigr)\notag\\
&\le\theta\,\kappa_n\bigl(\Sigma-\Theta\bigr)\notag\\
&=-\theta\,\kappa_n\bigl(\Theta-\Sigma\bigr)\notag\\
&\le-\kappa_1\kappa_n\bigl(\Theta-\Sigma\bigr)|df_H|^2\notag\\
&\le-\alpha_0|df_H|^2 .
\end{align}
Since $f$ is nonconstant,
\[
\widetilde{\mathcal T}(f)\le-\alpha_0\int_Me^{c}|df_H|^2\,dv_g<0,
\qquad
C_0=\kappa_{\max}^2\int_Me^{c}|df_H|^2\,dv_g ,
\]
and Theorem \ref{thm:index} gives $\ind(f)\ge\lceil\alpha_0/\kappa_{\max}^2\rceil$.
\end{proof}

\section{Exact Computation on a Heisenberg Nilmanifold}\label{sec:heis}

In this section the parallel test space is used in the case $k(f)=n$, in which the pullback bundle admits a global parallel orthonormal frame and Theorem \ref{thm:reduction} computes the index exactly.

\subsection{The Nilmanifold and its Sub-Laplacian}

Let $\mathbb H^3=(\mathbb R^3,\ast)$ be the Heisenberg group with group law
\[
(x,y,z)\ast(x',y',z')=\Bigl(x+x',\ y+y',\ z+z'+\tfrac12\bigl(xy'-yx'\bigr)\Bigr),
\]
and left-invariant frame
\[
X_1=\partial_x-\tfrac y2\,\partial_z,\qquad X_2=\partial_y+\tfrac x2\,\partial_z,\qquad X_3=\partial_z ,
\]
which satisfies $[X_1,X_2]=X_3$ and $[X_1,X_3]=[X_2,X_3]=0$. Let $g$ be the left-invariant metric making $\{X_1,X_2,X_3\}$ orthonormal, a distinguished Riemannian extension in the sense of Remark \ref{rem:extension}, and put $H:=\operatorname{span}\{X_1,X_2\}$. It is bracket-generating of rank $m=2$ and step $\mathfrak s=2$, the bracket $[X_1,X_2]$ having length $2$; moreover $\dim H=2$ and $\dim(H+[H,H])=3$ at every point, so $H$ is equiregular. Here $\mathcal V=\operatorname{span}\{X_3\}$.

\begin{lemma}\label{lem:lattice}
The subset $\Gamma:=\bigl\{(p,q,s)\in\mathbb H^3:\ p,q\in\mathbb Z,\ s\in\tfrac12\mathbb Z\bigr\}$ is a cocompact lattice in $\mathbb H^3$.
\end{lemma}

\begin{proof}
For $(p,q,s),(p',q',s')\in\Gamma$, the third coordinate of the product is $s+s'+\frac12(pq'-qp')$. Since $pq'-qp'\in\mathbb Z$, it lies in $\frac12\mathbb Z$, so $\Gamma$ is closed under multiplication. Moreover
\[
(p,q,s)\ast(-p,-q,-s)=\bigl(0,0,\tfrac12(-pq+qp)\bigr)=(0,0,0),
\]
so inverses lie in $\Gamma$ and $\Gamma$ is a subgroup. It is discrete, being contained in the discrete subset $\mathbb Z\times\mathbb Z\times\frac12\mathbb Z$ of $\mathbb R^3$.

For cocompactness, let $(x,y,z)\in\mathbb H^3$. Left translation by $(p,q,0)$ with $p=-\lfloor x\rfloor$ and $q=-\lfloor y\rfloor$ moves the first two coordinates into $[0,1)^2$. Left translation by $(0,0,s)$ with $s\in\frac12\mathbb Z$ acts by $(x,y,z)\mapsto(x,y,z+s)$ and fixes the first two coordinates; a suitable $s$ moves the third coordinate into $[0,\frac12)$. Hence every left coset $\Gamma\ast(x,y,z)$ has a representative in the compact box $[0,1]^2\times[0,\frac12]$.
\end{proof}

Let $M:=\Gamma\backslash\mathbb H^3$, a closed connected $3$-manifold, and let $\pi:\mathbb H^3\to M$ be the canonical projection, which is surjective. The fields $X_1,X_2,X_3$ are left invariant, hence descend to $M$, where they are denoted by the same symbols. Take $e_1=X_1$ and $e_2=X_2$. For left-invariant orthonormal fields the Koszul formula reads
\[
\langle\nabla_XY,Z\rangle=\tfrac12\bigl(\langle[X,Y],Z\rangle-\langle[Y,Z],X\rangle+\langle[Z,X],Y\rangle\bigr).
\]
With $X=Y=X_i$ this gives $\langle\nabla_{X_i}X_i,Z\rangle=-\langle[X_i,Z],X_i\rangle$. For $i=1,2$, the bracket $[X_i,Z]$ is a multiple of $X_3\perp X_i$, and for $i=3$ it vanishes. Hence
\[
\nabla_{X_1}X_1=\nabla_{X_2}X_2=\nabla_{X_3}X_3=0 ,
\]
so $\pi_H(\nabla_{e_i}e_i)=0$ and $\zeta=0$, and \eqref{eq:tension} simplifies to
\begin{equation}\label{eq:tension-heis}
\tau_H(f)=\beta_H(f)(X_1,X_1)+\beta_H(f)(X_2,X_2)+df_H\bigl(\nabla^Hc\bigr).
\end{equation}

Lemma \ref{lem:2.1} with $\psi\equiv1$ and $\alpha$ dual to $uX_i$, $i=1,2$, gives $\int_MX_iu\,dv_g=0$ for all $u\in C^\infty(M)$, that is, $X_i^*=-X_i$ on $L^2(M,dv_g)$. Consequently
\begin{equation}\label{eq:sublaplacian}
\int_M\bigl(-(X_1^2+X_2^2)u\bigr)\,u\,dv_g=\int_M|\nabla^Hu|^2\,dv_g,\qquad u\in C^\infty(M).
\end{equation}
The form $u\mapsto\int_M|\nabla^Hu|^2dv_g$ on $C^\infty(M)$ is closable, with closure defined on $\mathcal H^1_H(M)$. Its associated nonnegative self-adjoint operator, the Friedrichs extension of $-(X_1^2+X_2^2)$, is denoted by $\mathcal L$ and called the sub-Laplacian. By Lemma \ref{lem:compact-embed} applied to the trivial line bundle, $\mathcal L$ has compact resolvent. Hence its spectrum is a sequence of eigenvalues of finite multiplicity
\[
0=\lambda_0<\lambda_1\le\lambda_2\le\cdots\longrightarrow\infty .
\]
An eigenfunction $\phi$ with eigenvalue $\lambda_j$ satisfies $(X_1^2+X_2^2+\lambda_j)\phi=0$ in the sense of distributions, so it is smooth by H\"ormander's hypoellipticity theorem \cite{Hor67}. If $\mathcal L\phi=0$, then $\int_M|\nabla^H\phi|^2dv_g=0$, so $\phi$ is constant by Lemma \ref{lem:horiz-const}. Thus $\lambda_0=0$ is simple. For $\mu>0$ put
\[
N(\mu):=\#\{j:\lambda_j<\mu\},
\qquad
\mathcal E_{<\mu}:=\operatorname{span}\{\text{eigenfunctions with eigenvalue }<\mu\},
\]
so that $\dim\mathcal E_{<\mu}=N(\mu)$, and let $\mult_{\mathcal L}(\mu)$ be the multiplicity of $\mu$ as an eigenvalue, zero if $\mu\notin\spec\mathcal L$.

For $(j,k)\in\mathbb Z^2$, the functions $\cos(2\pi(jx+ky))$ and $\sin(2\pi(jx+ky))$ on $\mathbb H^3$ are invariant under left translation by $\Gamma$, since such a translation changes $jx+ky$ by an integer. Hence they descend to $M$. They are independent of $z$, so $X_1$ and $X_2$ act on them as $\partial_x$ and $\partial_y$, and
\begin{equation}\label{eq:torus-modes}
\mathcal L\cos\bigl(2\pi(jx+ky)\bigr)=4\pi^2(j^2+k^2)\cos\bigl(2\pi(jx+ky)\bigr),
\end{equation}
and likewise for the sine.

\subsection{The Explicit Family}

Maps of the form $\gamma\circ s$, with $\gamma$ a geodesic of the target and $s$ a suitable function on the source, were shown to be exponentially subelliptic harmonic in \cite[\S2.3, Example 4]{CDE20}, where the construction is attributed to Duan. The family below is a special case, adapted to the compact quotient. Our purpose is the computation of its index and nullity.

\begin{theorem}\label{thm:heis-family}
Let $a,b\in\mathbb Z$ with $(a,b)\ne(0,0)$, let $n\ge2$ and $r>0$, and let $u,v\in\mathbb R^{n+1}$ be orthonormal. Define
\[
\gamma:\mathbb R\to S^n(r),\qquad
\gamma(t):=r\bigl(\cos(2\pi t)\,u+\sin(2\pi t)\,v\bigr),
\qquad
\ell:\mathbb H^3\to\mathbb R,\quad \ell(x,y,z):=ax+by .
\]
Then $\gamma\circ\ell$ descends to a smooth map $f:M\to S^n(r)$ with $f\circ\pi=\gamma\circ\ell$, and:
\begin{enumerate}
\item[\rm(i)] $f$ is nonconstant, and $f(M)=S^n(r)\cap P$ is a great circle, where $P:=\operatorname{span}\{u,v\}$;
\item[\rm(ii)] $df(X_1)=a\,\gamma'(\ell)$, $df(X_2)=b\,\gamma'(\ell)$ and $df(X_3)=0$; in particular $|df_H|^2\equiv\Sigma:=4\pi^2r^2(a^2+b^2)$ is constant;
\item[\rm(iii)] $f$ is exponentially subelliptic harmonic;
\item[\rm(iv)] $E_H(f)=\operatorname{vol}_g(M)\exp\bigl(2\pi^2r^2(a^2+b^2)\bigr)$.
\end{enumerate}
\end{theorem}

\begin{proof}
The function $\ell$ is not defined on $M$. All differential computations are performed on the covering group $\mathbb H^3$. Since $X_1,X_2,X_3$ are left invariant and the resulting expressions, such as $\gamma'(\ell)$, are $\Gamma$-invariant by the periodicity of $\gamma$, they descend to $M$.

(i) For $\lambda=(p,q,s)\in\Gamma$, the left translation $L_\lambda(x,y,z)=\lambda\ast(x,y,z)$ has first two coordinates $x+p$ and $y+q$, so
\[
\ell\circ L_\lambda=\ell+(ap+bq),\qquad ap+bq\in\mathbb Z .
\]
Since $\gamma$ is $1$-periodic, $\gamma\circ\ell\circ L_\lambda=\gamma\circ\ell$ for every $\lambda\in\Gamma$. Hence $\gamma\circ\ell$ is $\Gamma$-invariant and descends to a unique smooth map $f$ on $M$ with $f\circ\pi=\gamma\circ\ell$. Since $\pi$ is surjective, $f(M)=(\gamma\circ\ell)(\mathbb H^3)$, and since $\ell$ is onto $\mathbb R$, this equals $\gamma(\mathbb R)=S^n(r)\cap P$. As $\gamma$ is injective on $[0,1)$, $f$ is nonconstant.

(ii) Since $\ell$ depends only on $x$ and $y$,
\[
X_1\ell=\partial_x\ell-\tfrac y2\partial_z\ell=a,\qquad
X_2\ell=\partial_y\ell+\tfrac x2\partial_z\ell=b,\qquad
X_3\ell=\partial_z\ell=0 .
\]
By the chain rule applied to $f\circ\pi=\gamma\circ\ell$,
\begin{align}
df(X_1)&=(X_1\ell)\,\gamma'(\ell)=a\,\gamma'(\ell),\\
df(X_2)&=(X_2\ell)\,\gamma'(\ell)=b\,\gamma'(\ell),\\
df(X_3)&=(X_3\ell)\,\gamma'(\ell)=0 .
\end{align}
Since $\gamma'(t)=2\pi r\bigl(-\sin(2\pi t)u+\cos(2\pi t)v\bigr)$ has constant norm $2\pi r$,
\[
|df_H|^2=|df(X_1)|^2+|df(X_2)|^2=(a^2+b^2)\,|\gamma'(\ell)|^2=4\pi^2r^2(a^2+b^2).
\]

(iii) By (ii), $c$ is constant, so $\nabla^Hc=0$ and $\zeta=0$. By \eqref{eq:tension-heis} it suffices to show $\beta_H(f)(X_1,X_1)+\beta_H(f)(X_2,X_2)=0$. As $\nabla_{X_i}X_i=0$, formula \eqref{eq:betaH} and the chain rule give, for $i=1,2$,
\begin{align}
\beta_H(f)(X_i,X_i)
&=\tilde\nabla_{X_i}\bigl(df(X_i)\bigr)\notag\\
&=\tilde\nabla_{X_i}\Bigl(\bigl(X_i\ell\bigr)\,\gamma'(\ell)\Bigr)\notag\\
&=\bigl(X_iX_i\ell\bigr)\,\gamma'(\ell)+\bigl(X_i\ell\bigr)^2\,\tilde\nabla_{\gamma'}\gamma' .
\end{align}
Now $X_1X_1\ell=X_1(a)=0$ and $X_2X_2\ell=X_2(b)=0$, while $\tilde\nabla_{\gamma'}\gamma'=0$ because $\gamma$ is a constant-speed great circle, hence a geodesic of $S^n(r)$. Therefore $\tau_H(f)=0$.

(iv) Immediate from (ii).
\end{proof}

\subsection{Reduction and Proof of Theorem \ref{thm:D}}

Let $\mathbf t:=\gamma'(\ell)/(2\pi r)$, a unit section of $f^{-1}TS^n(r)$ tangent to $f(M)$. Let $w_1,\dots,w_{n-1}$ be an orthonormal basis of $P^\perp$ with corresponding constant sections $W_1,\dots,W_{n-1}$, and put
\[
D:=aX_1+bX_2,
\qquad
\lambda:=\frac{\Sigma}{r^2}=4\pi^2(a^2+b^2).
\]

\begin{proposition}\label{prop:heis-split}
$\{\mathbf t,W_1,\dots,W_{n-1}\}$ is a global parallel orthonormal frame of $f^{-1}TS^n(r)$, so $k(f)=n$. For $v=(u_0,v_1,\dots,v_{n-1})$, with $u_0$ the coefficient of $\mathbf t$, the reduced form \eqref{eq:reduced} is
\begin{equation}\label{eq:heis-split}
e^{-\Sigma/2}\,\mathcal Q_f(v)=\mathcal Q_0(u_0)+\sum_{\alpha=1}^{n-1}\mathcal Q_1(v_\alpha),
\end{equation}
where
\[
\mathcal Q_0(u):=\int_M\Bigl[|\nabla^Hu|^2+4\pi^2r^2(Du)^2\Bigr]dv_g,
\qquad
\mathcal Q_1(v):=\int_M\Bigl[|\nabla^Hv|^2-\lambda v^2\Bigr]dv_g .
\]
Thus the form consists of a nonnegative tangential block with modified principal part and $n-1$ copies of the scalar Schr\"odinger form of $\mathcal L-\lambda$.
\end{proposition}

\begin{proof}
The computations are carried out on $\mathbb H^3$, as in the proof of Theorem \ref{thm:heis-family}. Since $\gamma$ is a geodesic, $\tilde\nabla_{\gamma'}\gamma'=0$, so for every $Y\in\Gamma(TM)$
\[
\tilde\nabla_Y\mathbf t=\frac{(Y\ell)}{2\pi r}\,\tilde\nabla_{\gamma'}\gamma'=0 ,
\]
and $W_\alpha\in P(f)$ by Proposition \ref{prop:subsphere} with $k=1$. As $\mathbf t$ takes values in $P$ and $w_\alpha\in P^\perp$, the frame is orthonormal, and it spans $f^{-1}TS^n(r)$ since the fibres have dimension $n$.

We compute the data \eqref{eq:xi-rho}. By Theorem \ref{thm:heis-family}(ii), $df(X_i)=2\pi r\,(X_i\ell)\,\mathbf t$. Hence
\[
\bigl\langle\mathbf t,df(X_i)\bigr\rangle=2\pi r\,(X_i\ell),
\qquad
\bigl\langle W_\alpha,df(X_i)\bigr\rangle=0,
\]
so that
\[
\xi_{\mathbf t}=2\pi r\bigl(aX_1+bX_2\bigr)=2\pi r\,D,
\qquad
\xi_\alpha=0,
\qquad
\bigl\langle\nabla^Hu_0,\xi_{\mathbf t}\bigr\rangle=2\pi r\,Du_0 .
\]
For the curvature terms, $\tilde R^N(X,Y)Z=r^{-2}(\langle Y,Z\rangle X-\langle X,Z\rangle Y)$ gives, for $V,V'\in T_{f(x)}S^n(r)$,
\[
\mathcal R_T(V,V')=\frac1{r^2}\sum_{i=1}^2\Bigl[|df(X_i)|^2\langle V,V'\rangle-\langle df(X_i),V\rangle\langle df(X_i),V'\rangle\Bigr]
=\frac{\Sigma}{r^2}\Bigl[\langle V,V'\rangle-\langle\mathbf t,V\rangle\langle\mathbf t,V'\rangle\Bigr],
\]
using $\sum_i|df(X_i)|^2=\Sigma$ and $\sum_i\langle df(X_i),V\rangle\langle df(X_i),V'\rangle=\Sigma\langle\mathbf t,V\rangle\langle\mathbf t,V'\rangle$. Consequently
\[
\rho_{\mathbf t\mathbf t}=0,
\qquad
\rho_{\mathbf t\alpha}=0,
\qquad
\rho_{\alpha\beta}=\lambda\,\delta_{\alpha\beta}.
\]
Inserting these data into \eqref{eq:reduced}, and using that $e^{c}=e^{\Sigma/2}$ is constant, gives \eqref{eq:heis-split}.
\end{proof}

\begin{theorem}\label{thm:heis-spectral}
Let $n\ge2$ and let $f$ be as in Theorem \ref{thm:heis-family}. Then
\[
\ind(f)=(n-1)\,N(\lambda),
\qquad
\nul(f)=1+(n-1)\,\mult_{\mathcal L}(\lambda),
\qquad
\mult_{\mathcal L}(\lambda)\ge4 .
\]
In particular $\ind(f)\ge n-1$, $\nul(f)\ge4n-3$, and $f$ is unstable, for every $r>0$.
\end{theorem}

\begin{proof}
Since $k(f)=n$, Theorem \ref{thm:reduction} and Proposition \ref{prop:heis-split} show that $\ind(f)$ is the index of the form $\mathcal Q_0\oplus\mathcal Q_1^{\oplus(n-1)}$ on smooth tuples, and that $\nul(f)$ is the dimension of its radical on $\mathcal H^1_H(M)^{\oplus n}$. We treat the index, the nullity and the multiplicity in turn.

\emph{Index.} Let $\{\phi_j\}$ be an $L^2(M,dv_g)$-orthonormal basis of eigenfunctions of $\mathcal L$, with $\mathcal L\phi_j=\lambda_j\phi_j$. For $v\in\mathcal E_{<\lambda}\setminus\{0\}$, write $v=\sum_{\lambda_j<\lambda}\hat v_j\phi_j$. By \eqref{eq:sublaplacian},
\[
\mathcal Q_1(v)=\sum_{j:\lambda_j<\lambda}\bigl(\lambda_j-\lambda\bigr)|\hat v_j|^2\ <\ 0 .
\]
Hence the form is negative definite on the space
\[
\mathcal W_-:=\bigl\{(0,v_1,\dots,v_{n-1}):\ v_\alpha\in\mathcal E_{<\lambda}\bigr\}
\]
of smooth tuples. This space has dimension $(n-1)N(\lambda)$, and $\ind(f)\ge(n-1)N(\lambda)$.

Conversely, consider the subspace of tuples $(u_0,v_1,\dots,v_{n-1})$ with $\langle v_\alpha,\phi_j\rangle_{L^2}=0$ for every $\alpha$ and every $j$ with $\lambda_j<\lambda$, the coefficient $u_0$ being unrestricted. It is defined by $(n-1)N(\lambda)$ linear conditions, so it has codimension at most $(n-1)N(\lambda)$. For $v\perp\mathcal E_{<\lambda}$ the spectral theorem gives $\int_M|\nabla^Hv|^2dv_g\ge\lambda\int_Mv^2dv_g$, whether or not $\lambda$ is itself an eigenvalue. Thus $\mathcal Q_1(v)\ge0$. Together with $\mathcal Q_0\ge0$, the form is nonnegative on this subspace. A subspace of dimension exceeding $(n-1)N(\lambda)$ on which the form is negative definite would meet it nontrivially, which is impossible. Hence $\ind(f)\le(n-1)N(\lambda)$.

\emph{Nullity.} The radical of an orthogonal sum of forms is the sum of the radicals. The form $\mathcal Q_0$ is nonnegative, so by the Cauchy--Schwarz inequality its radical is $\{u:\mathcal Q_0(u)=0\}$. This set is contained in $\{u:\nabla^Hu=0\}=\ker\mathcal L$, which consists of the constants; conversely, constants lie in it. So the radical of $\mathcal Q_0$ is one dimensional. The form $\mathcal Q_1$ is the form of $\mathcal L-\lambda$, so its radical on $\mathcal H^1_H(M)$ is $\ker(\mathcal L-\lambda)$, of dimension $\mult_{\mathcal L}(\lambda)$. This gives the formula for $\nul(f)$. The tangential null direction $\mathbf t$ generates the variation $f_t=\gamma\circ(\ell+t/(2\pi r))$, along which $E_H$ is constant.

\emph{Multiplicity.} By \eqref{eq:torus-modes}, the four functions
\[
\cos\bigl(2\pi(ax+by)\bigr),\quad \sin\bigl(2\pi(ax+by)\bigr),\quad \cos\bigl(2\pi(-bx+ay)\bigr),\quad \sin\bigl(2\pi(-bx+ay)\bigr)
\]
are eigenfunctions of $\mathcal L$ with eigenvalue $4\pi^2(a^2+b^2)=4\pi^2((-b)^2+a^2)=\lambda$. The frequency vectors $(a,b)$ and $(-b,a)$ are nonzero and orthogonal, so neither equals $\pm$ the other, and the four real Fourier modes are mutually orthogonal in $L^2$, hence linearly independent. Thus $\mult_{\mathcal L}(\lambda)\ge4$, and $\nul(f)\ge1+4(n-1)=4n-3$.

Finally, $\lambda_0=0<\lambda$ gives $N(\lambda)\ge1$, so $\ind(f)\ge n-1\ge1$.
\end{proof}

This proves Theorem \ref{thm:D}.

\begin{remark}\label{rem:killing-check}
The isometries of the target account for part, but not all, of the null space. For $A\in\mathfrak{so}(n+1)$ the field $K_A(q)=Aq$ is a Killing field of $S^n(r)$, so each $K_A\circ f$ is a null direction by Proposition \ref{prop:killing}. For $w\in P^\perp$ and $A=w\otimes u^\flat-u\otimes w^\flat$ one has $K_A(q)=\langle u,q\rangle w-\langle w,q\rangle u$. Since $\langle w,f\rangle=0$,
\[
K_A\circ f=\langle u,f\rangle\,w=r\cos(2\pi\ell)\,w .
\]
Replacing $u$ by $v$ gives $r\sin(2\pi\ell)\,w$, and $A=v\otimes u^\flat-u\otimes v^\flat$ gives
\[
K_A\circ f=\langle u,f\rangle v-\langle v,f\rangle u=r\bigl(\cos(2\pi\ell)v-\sin(2\pi\ell)u\bigr)=r\,\mathbf t .
\]
These sections span a subspace of dimension $1+2(n-1)=2n-1$, which equals $\dim SO(n+1)-\dim SO(n-1)$, the dimension of the orbit of $f$ under the isometry group of $S^n(r)$. By Theorem \ref{thm:heis-spectral}, however, $\nul(f)\ge4n-3>2n-1$. The additional null directions $\cos(2\pi(-bx+ay))\,w$ and $\sin(2\pi(-bx+ay))\,w$, $w\in P^\perp$, come from the rotated frequency $(-b,a)$ and are not produced by isometries of the target. We do not claim that they integrate to families of critical maps. The Killing directions also explain why $N(\lambda)$ is defined with a strict inequality: eigenvalues equal to $\lambda$ contribute to the nullity, not to the index.
\end{remark}

\begin{remark}\label{rem:exact-mult}
The spectrum of the sub-Laplacian on Heisenberg nilmanifolds is known explicitly. See Thangavelu \cite[\S4.1]{Tha09}, whose group law, lattice $\mathbb Z^{2}\times\frac12\mathbb Z$ and normalization of the sub-Laplacian agree with ours up to the names of the coordinates and the sign of $[X_1,X_2]$, which does not affect the spectrum. For the lattice $\Gamma$ above, $L^2(M)$ decomposes according to the characters of the centre. The functions independent of $z$ give the eigenvalues $4\pi^2(j^2+k^2)$ with $(j,k)\in\mathbb Z^2$, by \eqref{eq:torus-modes}. Since invariance under $(0,0,\frac12)$ forces period $\frac12$ in $z$, the other characters are $z\mapsto e^{4\pi imz}$ with $m\in\mathbb Z\setminus\{0\}$, on which $X_3$ acts as $4\pi im$, and they give Landau-type eigenvalues $4\pi|m|(2q+1)$, $q\ge0$. These two families never meet at $\lambda=4\pi^2(a^2+b^2)$, since $4\pi^2d=4\pi|m|(2q+1)$ with integers $d>0$ would give $\pi=|m|(2q+1)/d\in\mathbb Q$. With this spectral decomposition one obtains
\[
\mult_{\mathcal L}(\lambda)=r_2(a^2+b^2),\qquad r_2(d):=\#\{(j,k)\in\mathbb Z^2:j^2+k^2=d\},
\]
and hence $\nul(f)=1+(n-1)\,r_2(a^2+b^2)$. For example, $(a,b)=(1,0)$ gives $r_2(1)=4$ and $\nul(f)=4n-3$. This refinement relies on the cited spectral decomposition and is not used elsewhere.
\end{remark}

\begin{corollary}\label{cor:heis-lattice}
With $f$ and $\lambda$ as above,
\[
\ind(f)\ \ge\ (n-1)\cdot\#\bigl\{(j,k)\in\mathbb Z^2:\ j^2+k^2<a^2+b^2\bigr\}.
\]
In particular $\ind(f)\ge5(n-1)$ whenever $a^2+b^2\ge2$.
\end{corollary}

\begin{proof}
By \eqref{eq:torus-modes}, for $(j,k)\in\mathbb Z^2$ the functions $\cos(2\pi(jx+ky))$ and $\sin(2\pi(jx+ky))$ are eigenfunctions of $\mathcal L$ with eigenvalue $4\pi^2(j^2+k^2)$. The lattice point $(0,0)$ contributes the constants, of real dimension $1$. Each pair $\{(j,k),(-j,-k)\}$ with $(j,k)\ne(0,0)$ contributes two independent real functions, and these are orthogonal in $L^2$ to those of other pairs. Hence the real span of these functions over all lattice points with $j^2+k^2<a^2+b^2$ has dimension equal to the number of such points, and $N(\lambda)$ is at least that number. For $a^2+b^2\ge2$ these points include $(0,0),(\pm1,0),(0,\pm1)$. Theorem \ref{thm:heis-spectral} concludes.
\end{proof}

\begin{remark}\label{rem:ext-matrix}
For this family the coefficient matrix $\mathbb I$ of Theorem \ref{thm:index} can be computed explicitly, which separates the two sources of loss in the averaging bound. Here $p=1$. Take the ambient orthonormal basis $u,v,w_1,\dots,w_{n-1}$ of $\mathbb R^{n+1}$, and put $\theta:=2\pi\ell$ and $V_M:=\operatorname{vol}_g(M)$. Since $f=r(\cos\theta\,u+\sin\theta\,v)$, the tangential projections along $f$ are
\[
u^\top=u-\cos\theta\,\bigl(\cos\theta\,u+\sin\theta\,v\bigr)=-\sin\theta\,\mathbf t,
\qquad
v^\top=\cos\theta\,\mathbf t,
\qquad
w_\alpha^\top=W_\alpha .
\]
By \eqref{eq:heis-split}, $I_H(u^\top,u^\top)=e^{\Sigma/2}\mathcal Q_0(\sin\theta)$. Now $\nabla^H\sin\theta=\cos\theta\,\nabla^H\theta$ with $|\nabla^H\theta|^2=\lambda$, and $D\sin\theta=2\pi(a^2+b^2)\cos\theta$, so $4\pi^2r^2(D\sin\theta)^2=\lambda\Sigma\cos^2\theta$. Since $\cos^2\theta$ and $\sin^2\theta$ have mean $\frac12$ over $M$ and $\cos\theta\sin\theta$ has mean zero, one finds
\[
\mathcal Q_0(\sin\theta)=\mathcal Q_0(\cos\theta)=\tfrac12\lambda(1+\Sigma)V_M,
\]
while the mixed tangential entry vanishes. The tangential--normal entries vanish by \eqref{eq:heis-split}, and $I_H(W_\alpha,W_\beta)=-e^{\Sigma/2}\lambda V_M\delta_{\alpha\beta}$ by \eqref{eq:subsphere-index}. Hence
\[
\mathbb I=e^{\Sigma/2}\lambda V_M\,\diag\Bigl(\tfrac{1+\Sigma}2,\ \tfrac{1+\Sigma}2,\ \underbrace{-1,\dots,-1}_{n-1}\Bigr),
\qquad
\tr\mathbb I=e^{\Sigma/2}\lambda V_M\bigl(\Sigma-(n-2)\bigr),
\]
in agreement with \eqref{eq:master} and \eqref{eq:Qsphere}. The matrix $\mathbb I$ has exactly $n-1$ negative eigenvalues for every $r>0$. Using only the trace loses these once the two positive tangential eigenvalues balance or outweigh them, that is, once $\Sigma\ge n-2$. Using only the finite-dimensional extrinsic test space loses the difference between $n-1$ and the full index $(n-1)N(\lambda)$.
\end{remark}

\begin{remark}\label{rem:weyl}
The eigenfunctions of Corollary \ref{cor:heis-lattice} are those independent of $z$, and account only for part of the spectrum. For an equiregular sub-Laplacian on a closed manifold, M\'etivier \cite{Met76} proved the Weyl law $N(\mu)\sim C\mu^{Q/2}$ as $\mu\to\infty$, where $Q$ is the homogeneous dimension and $C>0$ is determined by the sub-Riemannian geometry and the normalization of the operator. Multiplying $dv_g$ by a positive constant changes neither the operator nor its eigenvalues. The law applies here because $H$ is equiregular, with $\dim H=2$ and $\dim(H+[H,H])=3$ at every point, so that $Q=1\cdot2+2\cdot1=4$. Hence
\[
\ind(f)=(n-1)N(\lambda)\ \sim\ (n-1)\,C\,\lambda^{2}=(n-1)\,C\,(4\pi^2)^2\bigl(a^2+b^2\bigr)^2
\qquad\text{as }a^2+b^2\to\infty .
\]
For fixed $n\ge3$ and fixed $r>0$, the averaging bound of Corollary \ref{cor:index-sphere} is vacuous in the same regime, by Remark \ref{rem:r-bound}. If $r$ is allowed to vary with $(a,b)$, for instance so that $\Sigma$ stays fixed, this need not be the case. The exponent is governed by half the homogeneous dimension rather than half the topological dimension, a feature specific to the sub-Riemannian setting.
\end{remark}

\begin{remark}\label{rem:r-bound}
The normal blocks of \eqref{eq:heis-split} do not involve $r$, since $\lambda=4\pi^2(a^2+b^2)$. The radius enters only through the positive constant $e^{\Sigma/2}$ and the nonnegative tangential term $4\pi^2r^2(Du_0)^2$. Hence, under the trivialization $f^{-1}TS^n(r)\cong M\times\mathbb R^n$ furnished by the parallel frame, the negative spectral subspace $\mathcal W_-$ of the proof of Theorem \ref{thm:heis-spectral} corresponds to the same space of coefficient functions for all $r>0$. By contrast, Corollary \ref{cor:index-sphere} gives here
\[
\ind(f)\ \ge\ \max\bigl\{0,\ \lceil(n-2)-\Sigma\rceil\bigr\},
\]
which decreases in $r$ and is vacuous once $\Sigma=4\pi^2r^2(a^2+b^2)\ge n-2$. For $n\ge3$ and
\[
r<\frac{\sqrt{n-2}}{2\pi\sqrt{a^2+b^2}}
\]
the pointwise hypothesis of Theorem \ref{thm:sphere} is satisfied. This range is nonempty, so the hypotheses of the instability and rigidity statements there are realized by nonconstant maps. The equality in Theorem \ref{thm:reduction} uses $k(f)=n$, which holds here because the image is a great circle of a round sphere and a global parallel frame has been exhibited. For a closed geodesic in a general target the normal holonomy may be nontrivial. If the image spans a totally geodesic $S^k(r)$ with $k\ge2$, Proposition \ref{prop:subsphere} gives only $k(f)\ge n-k$ and the reduction yields a lower bound.
\end{remark}

\appendix

\section{The Jacobi Operator}\label{sec:jacobi}

We record the differential operator that represents the index form, and compute its principal symbol. Recall $\mathcal F$ from Corollary \ref{cor:F-vanishes}, and put
\[
\Delta_HW:=\sum_{i=1}^m\Bigl(\tilde\nabla_{e_i}\tilde\nabla_{e_i}W-\tilde\nabla_{\pi_H(\nabla_{e_i}e_i)}W\Bigr).
\]

\begin{theorem}\label{thm:jacobi-explicit}
Let $f$ be exponentially subelliptic harmonic and define
\begin{align}
L_HW:={}&-\Delta_HW+\tilde\nabla_\zeta W-\tilde\nabla_{\nabla^Hc}W\notag\\
&-df_H\Bigl(\nabla^H\mathcal F(W)\Bigr)+\sum_{i=1}^m\tilde R^N\bigl(df(e_i),W\bigr)df(e_i).\label{eq:LH}
\end{align}
Then $I_H(W,W')=\int_Me^{c}\langle L_HW,W'\rangle\,dv_g$ for all $W,W'\in\Gamma(f^{-1}TN)$, so $L_H$ is formally symmetric with respect to $e^{c}dv_g$.
\end{theorem}

\begin{proof}
We transform the three terms of \eqref{eq:index-polarized} in turn.

\emph{Curvature term.} For fixed $X,W$ the endomorphism $\tilde R^N(X,W)$ is skew-adjoint, that is, $\tilde R_N(X,W,Z,T)=-\tilde R_N(X,W,T,Z)$. Hence
\[
-\tilde R_N\bigl(df(e_i),W,W',df(e_i)\bigr)=\tilde R_N\bigl(df(e_i),W,df(e_i),W'\bigr),
\]
and summing over $i$ turns $-\mathcal R_T(W,W')$ into $\bigl\langle\sum_i\tilde R^N(df(e_i),W)df(e_i),W'\bigr\rangle$.

\emph{Second-order term.} Apply Lemma \ref{lem:2.1} with $\psi=e^{c}$ and the horizontal one-form $\alpha(X):=\langle\tilde\nabla_{\pi_HX}W,W'\rangle$. Metric compatibility gives
\[
e_i\alpha(e_i)=\bigl\langle\tilde\nabla_{e_i}\tilde\nabla_{e_i}W,W'\bigr\rangle+\bigl\langle\tilde\nabla_{e_i}W,\tilde\nabla_{e_i}W'\bigr\rangle,
\qquad
\alpha\bigl(\pi_H(\nabla_{e_i}e_i)\bigr)=\bigl\langle\tilde\nabla_{\pi_H(\nabla_{e_i}e_i)}W,W'\bigr\rangle,
\]
so \eqref{eq:ibp} yields
\[
\int_Me^{c}\sum_{i=1}^m\bigl\langle\tilde\nabla_{e_i}W,\tilde\nabla_{e_i}W'\bigr\rangle\,dv_g
=\int_Me^{c}\bigl\langle-\Delta_HW+\tilde\nabla_\zeta W-\tilde\nabla_{\nabla^Hc}W,\ W'\bigr\rangle\,dv_g .
\]

\emph{Exponential coupling term.} Put $\varphi:=\mathcal F(W)\in C^\infty(M)$. For any $W'$, the Leibniz rule gives
\begin{align}
\mathcal F\bigl(\varphi W'\bigr)
&=\sum_{i=1}^m\bigl\langle(e_i\varphi)W'+\varphi\,\tilde\nabla_{e_i}W',\ df(e_i)\bigr\rangle\notag\\
&=\bigl\langle W',df_H\bigl(\nabla^H\varphi\bigr)\bigr\rangle+\varphi\,\mathcal F(W').
\end{align}
Applying Corollary \ref{cor:F-vanishes} to the section $\varphi W'$ gives
\[
0=\int_Me^{c}\mathcal F\bigl(\varphi W'\bigr)dv_g
=\int_Me^{c}\Bigl[\mathcal F(W)\mathcal F(W')+\bigl\langle W',df_H\bigl(\nabla^H\mathcal F(W)\bigr)\bigr\rangle\Bigr]dv_g ,
\]
that is,
\[
\int_Me^{c}\mathcal F(W)\mathcal F(W')\,dv_g=\int_Me^{c}\bigl\langle-df_H\bigl(\nabla^H\mathcal F(W)\bigr),W'\bigr\rangle\,dv_g .
\]

Adding the three contributions gives the identity. Since \eqref{eq:index-polarized} is symmetric in $W$ and $W'$, $L_H$ is formally symmetric with respect to $e^{c}dv_g$.
\end{proof}

\begin{remark}\label{rem:jacobi-link}
The computation shows that, with respect to $e^{c}dv_g$, the second-order part of $L_H$ is
\[
\nabla_H^*\nabla_H+\mathcal F^*\mathcal F ,
\]
a sum of two nonnegative operators. The first three terms of \eqref{eq:LH} constitute the horizontal rough Laplacian $\nabla_H^*\nabla_H$, and the coupling term $-df_H(\nabla^H\mathcal F(\cdot))$ is exactly $\mathcal F^*\mathcal F$. This accounts for the form of the principal symbol below, and for the fact that the coupling term may be discarded in Lemma \ref{lem:lower-bound} and Proposition \ref{prop:finiteness}(i). Moreover, the operator $A$ of Proposition \ref{prop:finiteness}(ii) is the Friedrichs realization of $L_H+C_1$ on $L^2_c$: by Theorem \ref{thm:jacobi-explicit}, $a$ is the quadratic form of $L_H+C_1$ on smooth sections, and $\mathcal H^1_H$ is by construction the closure of the smooth sections in the corresponding form norm.
\end{remark}

\begin{proposition}\label{prop:symbol}
Let $x\in M$ and $\xi\in T^*_xM$. Write $\xi_H:=\xi|_{H_x}$, let $\xi_H^\sharp\in H_x$ be its $g_H$-dual, and put $Z:=df_H(\xi_H^\sharp)\in T_{f(x)}N$. Then the principal symbol of $L_H$ at $(x,\xi)$, acting on $T_{f(x)}N$, is
\begin{equation}\label{eq:symbol}
\sigma_\xi(L_H)=|\xi_H|^2\operatorname{Id}+Z\otimes Z,
\qquad\text{that is,}\qquad
\sigma_\xi(L_H)W=|\xi_H|^2W+\langle W,Z\rangle Z .
\end{equation}
If $Z=0$, then $\sigma_\xi(L_H)=|\xi_H|^2\operatorname{Id}$. If $Z\ne0$, then $Z\otimes Z=|Z|^2\proj_{\mathbb RZ}$, so $\sigma_\xi(L_H)$ has eigenvalue $|\xi_H|^2+|Z|^2$ on $\mathbb RZ$ and eigenvalue $|\xi_H|^2$ on $Z^\perp$. In either case $\sigma_\xi(L_H)$ is positive definite if and only if $\xi_H\ne0$. Consequently $L_H$ is elliptic if and only if $H=TM$; when $\rank H<\dim M$, the principal symbol vanishes on $H^\circ\setminus\{0\}$ and is positive definite whenever $\xi|_H\ne0$.
\end{proposition}

\begin{proof}
Replace $\tilde\nabla_{e_i}$ by $\sqrt{-1}\,\xi(e_i)$ in \eqref{eq:LH} and retain the terms of order two. The operator $-\Delta_H$ contributes
\[
-\sum_{i=1}^m\bigl(\sqrt{-1}\,\xi(e_i)\bigr)^2\operatorname{Id}=\sum_{i=1}^m\xi(e_i)^2\operatorname{Id}=|\xi_H|^2\operatorname{Id},
\]
while $\tilde\nabla_\zeta W$, $\tilde\nabla_{\nabla^Hc}W$ and the curvature term are of order at most one. The map $W\mapsto\mathcal F(W)$ has symbol
\[
W\longmapsto\sqrt{-1}\sum_{i=1}^m\xi(e_i)\bigl\langle W,df(e_i)\bigr\rangle=\sqrt{-1}\,\langle W,Z\rangle,
\]
and $u\mapsto-df_H(\nabla^Hu)$ has symbol $u\mapsto-\sqrt{-1}\,u\,Z$. Composing gives $W\mapsto\langle W,Z\rangle Z$, which proves \eqref{eq:symbol}. If $Z\ne0$, then $\langle W,Z\rangle Z=|Z|^2\proj_{\mathbb RZ}W$, which gives the eigenvalue description. In both cases $Z\otimes Z$ is nonnegative of rank at most one. Hence $\sigma_\xi(L_H)$ is positive definite exactly when $|\xi_H|^2>0$. Finally, $\xi_H\ne0$ holds for every $\xi\ne0$ if and only if $H^\circ=0$, that is, if and only if $H=TM$.
\end{proof}

\bigskip
\noindent Xin Huang\\
School of Mathematics and Statistics\\
Nanjing University of Information Science and Technology\\
Nanjing 210044, P.\,R.\,China\\
\texttt{003941@nuist.edu.cn}

\end{document}